\documentclass{amsart}
\usepackage{amssymb,stmaryrd}
\usepackage{float}
\usepackage[inline]{enumitem}
\usepackage{tikz-cd}

\newtheorem{theorem}{Theorem}[section]
\newtheorem{claim}{Claim}[theorem]
\newtheorem{subclaim}{Subclaim}[claim]
\newtheorem{question}[theorem]{Question}
\newtheorem{lemma}[theorem]{Lemma}
\newtheorem{fact}[theorem]{Fact}	
\newtheorem{prop}[theorem]{Proposition}
\newtheorem{cor}[theorem]{Corollary}

\newtheorem*{thma}{Theorem~A}
\newtheorem*{thmb}{Theorem~B}
\newtheorem*{thmc}{Theorem~C}
\newtheorem*{thmd}{Theorem~D}

\theoremstyle{definition}
\newtheorem{definition}[theorem]{Definition}

\theoremstyle{remark}
\newtheorem{remark}[theorem]{Remark} 
\newtheorem{conv}[theorem]{Convention}

\DeclareMathOperator{\unif}{Unif}
\DeclareMathOperator{\reg}{Reg}
\DeclareMathOperator{\otp}{otp}
\DeclareMathOperator{\ch}{CH}
\DeclareMathOperator{\gch}{GCH}
\newcommand*\axiomfont[1]{\textsf{\textup{#1}}}
\newcommand\ma{\axiomfont{MA}}
\newcommand\pfa{\axiomfont{PFA}}
\newcommand\pid{\axiomfont{PID}}
\DeclareMathOperator{\hd}{hd}
\DeclareMathOperator{\Lin}{L}
\DeclareMathOperator{\acc}{acc}
\DeclareMathOperator{\nacc}{nacc}
\DeclareMathOperator{\zfc}{ZFC}
\DeclareMathOperator{\ad}{AD}
\DeclareMathOperator{\cl}{cl}
\DeclareMathOperator{\cf}{cf}
\DeclareMathOperator{\mup}{mup}
\DeclareMathOperator{\dom}{dom}
\DeclareMathOperator{\im}{Im}

\DeclareMathOperator{\h}{ht}
\DeclareMathOperator{\p}{P}
\newcommand*\sqleft[1]{\mathrel{_{#1}{\sqsubseteq}}}
\renewcommand\mid{\mathrel{|}\allowbreak}
\def\s{\subseteq}
\def\br{\blacktriangleright}

\title[A guessing principle from a Souslin tree]{A guessing principle from a Souslin tree,\\with applications to topology}
\author{Assaf Rinot}
\address{Department of Mathematics, Bar-Ilan University, Ramat-Gan 5290002, Israel.}
\urladdr{http://www.assafrinot.com}

\author{Roy Shalev}
\address{Department of Mathematics, Bar-Ilan University, Ramat-Gan 5290002, Israel.}
\urladdr{https://roy-shalev.github.io/}
\subjclass[2010]{Primary 03E05, 54G20; Secondary 03E35, 03E65}

\begin{document}
\dedicatory{This paper is dedicated to the memory of Kenneth Kunen (1943--2020)}
\begin{abstract}
	We introduce a new combinatorial principle which we call $\clubsuit_{\ad}$.
	This principle asserts the existence of a certain multi-ladder system with guessing and almost-disjointness features,
	and is shown to be sufficient for carrying out de Caux type constructions of topological spaces.

	Our main result states that strong instances of $\clubsuit_{\ad}$ follow from the existence of a Souslin tree.
	It is also shown that the weakest instance of $\clubsuit_{\ad}$ does not follow from the existence of an almost Souslin tree.

	As an application, we obtain a simple, de Caux type proof of Rudin's result that if there is a Souslin tree,
	then there is an $S$-space which is Dowker.
	\end{abstract}
\maketitle

\section{Introduction} 
All topological spaces under consideration are assumed to be $T_1$ and Hausdorff.

A \emph{Dowker space} is a normal topological space whose product with the unit interval is not normal.
Dowker \cite{Dowker_C.H.} raised the question of their very existence, and
gave a useful characterization of these spaces.
The first consistent example of such a space was given by Rudin \cite{Rudin_souslin_line_dowker_space}, 
who constructed a Dowker space of size $\aleph_1$, assuming the existence of a Souslin tree. 
Later on, in \cite{Rudin_first_Dowker_ZFC_example}, Rudin constructed another Dowker space, this time in ZFC, and of  cardinality $ (\aleph_\omega)^{\aleph_0}$.
Two decades later, Balogh \cite{Balogh_space} gave a $\zfc$ construction of a Dowker space of size $2^{\aleph_0}$, and Kojman and Shelah \cite{MR1605988} gave a $\zfc$ construction of a Dowker space of size $\aleph_{\omega+1}$.
A question remaining of focal interest ever since is whether $\zfc$ proves the existence of a \emph{small} Dowker space.
One of the sleekest consistent constructions of a Dowker space of size $\aleph_1$
may be found in de Caux's paper \cite{de_Caux_space}, assuming the combinatorial principle $\clubsuit$.
Whether $\clubsuit$ implies the existence of a Souslin tree
was asked by Juh\'asz around 1987 and remains open to this date.
For a comprehensive survey on Dowker spaces, we refer the reader to \cite{Rudin_chapter_handbook_of_set_topology}, \cite{Szepytycki_and_Weiss_review} and \cite{Dowker_Open_Quesions}.

An \emph{$S$-space} is a regular topological space which is hereditary separable but not hereditary Lindel\"of.
Whether such a space exists was asked at the late 1960's by Hajnal and Juh\'asz and independently by Countryman.
Along the years many consistent constructions of $S$-spaces were found,
many of which are due to Kunen and his co-authors  \cite{two_more_S_spaces,MR667660,MR1234990,MR2023413,MR2492307,MR3690760,MR3962617}.
Rudin showed that the existence of a Souslin tree yields an $S$-space \cite{rudin_S_space_Souslin},
and even an $S$-space which is Dowker \cite{rudin_separable_dowker_space}.
Juhász, Kunen, and Rudin \cite{two_more_S_spaces} gave an example from $\ch$ of a first countable, locally compact
$S$-space, known as the \emph{Kunen Line},
as well as an example from $\ch$ of a first countable, $S$-space which is Dowker.

In the other direction, Kunen \cite{MR0440487} proved that, assuming $\ma_{\aleph_1}$, there are no strong $S$-spaces (that is, spaces all of whose finite powers are $S$-spaces), 
Szentmikl\'{o}ssy \cite{MR588860} proved that, assuming $\ma_{\aleph_1}$, there are no compact $S$-spaces,
and Todor\v{c}evi\'{c} \cite{Stevo_No_S_Space} proved that, assuming $\pfa$, there are no $S$-spaces whatsoever.
For a comprehensive survey on $S$-spaces, we refer the reader to \cite{MR588816}, \cite{MR776626} and \cite{MR2713439}.

An \emph{$O$-space} is an uncountable regular topological space all of whose uncountable open sets are co-countable.
Note that any $O$-space is an $S$-space of size $\aleph_1$ all of whose closed sets are $G_\delta$ (i.e., the space is \emph{perfect}),
and hence not Dowker (cf. \cite[pp.~248]{MR467673}, \cite[\S2]{CH_with_no_Ostaszweski_spaces} and \cite[\S5]{MR1934262}).
This class of spaces is named after Adam Ostaszewski who constructed in \cite{Ostaszewski_space}, assuming $\clubsuit$, a normal, 
locally compact, non-Lindel\"of $O$-space. He also showed that, assuming $\ch$, the space can be made countably compact.
A few years later, Mohammed Dahroug, who was a Ph.D. student of William Weiss at the University of Toronto,
constructed a first-countable, locally compact,
non-Lindel\"of $O$-space from a Souslin tree.
He also showed that, assuming $\ch$, the space can be made countably compact and normal.
Dahroug's work was never typed down.

The purpose of this paper is to formulate a combinatorial principle that follows both from $\clubsuit$ and from the existence of a Souslin tree,
and is still strong enough to yield an $S$-space which is Dowker, as well as a normal $O$-space. We call it $\clubsuit_{\ad}$. The exact definition may be found in Definition~\ref{clubsuit AD definition} below.
The main results of this paper read as follows.
\begin{thma} \begin{enumerate}
\item For every infinite cardinal $\lambda$, if there exists a $\cf(\lambda)$-complete $\lambda^+$-Souslin tree,
then for every partition $\mathcal S$ of $E^{\lambda^+}_{\cf(\lambda)}$ into stationary sets,
$\clubsuit_{\ad}(\mathcal S,{<}{\cf(\lambda)})$ holds.
\item For every regular uncountable cardinal $\kappa$, if there exists a regressive $\kappa$-Souslin tree,
then for every partition $\mathcal S$ of $E^\kappa_\omega$ into stationary sets,
$\clubsuit_{\ad}(\mathcal S,{<}\omega)$ holds.
\end{enumerate}
\end{thma}

\begin{thmb} Suppose that $\mathcal S$ is an infinite partition of some non-reflecting stationary subset of a regular uncountable cardinal $\kappa$.
If $\clubsuit_{\ad}(\mathcal S,2)$ holds,
then there exists a Dowker space of cardinality $\kappa$.
\end{thmb}

Note that for every infinite regular cardinal $\lambda$, $E^{\lambda^+}_\lambda$ is a non-reflecting stationary subset of $\lambda^+$.

\begin{thmc} If $\clubsuit_{\ad}(\{\omega_1\},1)$ holds, then there exists a  collectionwise normal non-Lindel\"of $O$-space.
\end{thmc}

\begin{thmd} If $\clubsuit_{\ad}(\{E^{\lambda^+}_{\lambda}\},1)$ holds for an infinite regular cardinal $\lambda$,
then there exists a collectionwise normal Dowker space of cardinality $\lambda^+$, having hereditary density $\lambda$
and Lindel\"of degree $\lambda^+$.
\end{thmd}

\subsection{Organization of this paper} In Section~\ref{section2}, we formulate the guessing principle $\clubsuit_{\ad}(\mathcal S,{<}\theta)$ and its refinement $\clubsuit_{\ad}(\mathcal S,\mu,{<}\theta)$,
prove  that $\clubsuit(S)$ entails a strong instance of $\clubsuit_{\ad}(\mathcal S,{<}\omega)$,
and that it is consistent that  $\clubsuit_{\ad}(\{\omega_1\},{<}\omega)$ holds, but $\clubsuit(\omega_1)$ fails.
It is also shown that the weakest instance $\clubsuit_{\ad}(\{\kappa\},1,1)$ fails for $\kappa$ weakly compact,
and may consistently fail for $\kappa:=\omega_1$.
The proof of Theorem~A will be found there.

In Section~\ref{sectionladdersystemspace}, we present a $\clubsuit_{\ad}(\mathcal S,1,2)$-based construction of a Dowker space which is moreover a ladder-system space.
This covers scenarios previously considered by Good, Rudin and  Weiss, in which the Dowker spaces constructed were not ladder-system spaces.
The proof of Theorem~B will be found there.

In Section~\ref{Collectionwisesection}, we present two $\clubsuit_{\ad}(\{E^{\lambda^+}_\lambda\},\lambda,1)$-based constructions of collectionwise normal spaces of small hereditary density and large Lindel\"of degree.
These spaces are not ladder-system spaces, rather, they are de Caux type spaces.
The proof of Theorems C and D will be found there.

In Section~\ref{normalsquare}, we comment on a construction of Szeptycki of an $\aleph_2$-sized Dowker space with a normal square,
assuming that $ \diamondsuit^*(S) $ holds for some stationary $S\s E^{\omega_2}_{\omega_1}$.
Here, it is demonstrated that the construction may be carried out from a weaker assumption which is known to be consistent with the failure of $\clubsuit(E^{\omega_2}_{\omega_1})$.

\subsection{Notation}\label{notationsubsection}
The \emph{hereditary density number} of a topological space $\mathbb X$, denoted $\hd(\mathbb X)$, is the least infinite cardinal $\lambda$ such that each subspace of $\mathbb X$ contains a dense subset of cardinality at most $\lambda$.
The \emph{Lindel\"of degree} of $\mathbb X$, denoted $\Lin(\mathbb X)$ is the least infinite cardinal $\lambda$ such that every open cover of $\mathbb X$ has a subcover of cardinality at most $\lambda$.
For an accessible cardinal $\kappa$ and a cardinal $\lambda<\kappa$, we write $\log_\lambda(\kappa):=\min\{\chi\mid \lambda^\chi\ge\kappa\}$.
$\reg(\kappa)$ denotes the set of all infinite regular cardinal below $\kappa$.
For a set of ordinals $C$, we write $ \acc^+(C):=\{ \alpha<\sup(C) \mid  \sup(C\cap \alpha)=\alpha>0  \} $, 
$ \acc(C):=\{ \alpha\in C \mid  \sup(C\cap \alpha)=\alpha>0  \} $ and $\nacc(C):=C\setminus\acc(C)$.
For ordinals $ \alpha<\gamma$, denote $ E^{\gamma}_\alpha := \{\beta<\gamma \mid \cf(\beta)=\alpha \} $ and define $E^{\gamma}_{\neq\alpha}, E^{\gamma}_{<\alpha}, E^{\gamma}_{\le\alpha}, E^{\gamma}_{>\alpha}, E^{\gamma}_{\ge\alpha}$ similarly. 
For a set $A$ and a cardinal $\theta$, write $[A]^\theta := \{ B\subseteq A \mid |B|=\theta \}$ and define $[A]^{<\theta}$ similarly. 
For two sets $A$ and $B$, we write $A\s^* B$ to express that either $A=\emptyset$ or $A\setminus\alpha\s B$ for some $\alpha\in A$.
For a poset $(P,\lhd)$ and an element $x\in P$, we write $x_\downarrow:=\{ y\in P\mid y\lhd x\}$ and $x^\uparrow:=\{ y\in P\mid x\lhd y\}$.
For a family $\mathcal A$ of subsets of some ordinal, we let $\mup(\mathcal A):=\sup\{ \min(a) \mid a\in \mathcal A,~a\neq \emptyset \}$.

\subsection{Conventions} Throughout the paper, $\kappa$ stands for a regular uncountable cardinal,
and $\lambda$ stands for an infinite cardinal.

\section{A new guessing principle}\label{section2}
We commence by recalling some classic guessing principles.

\begin{definition}\label{principles} For a stationary subset $ S\subseteq \kappa $:
\begin{enumerate}
\item  	$ \diamondsuit^*(S) $ asserts the existence of a sequence $ \langle \mathcal A_\alpha \mid \alpha\in S \rangle $ such that:
	\begin{itemize}
		\item for all $ \alpha\in S $, $ \mathcal A_\alpha \subseteq \mathcal P( \alpha) $ and $ |\mathcal A_\alpha|\le|\alpha|$;
		\item for every $B\s \kappa$, there exists a club $ C\subseteq \kappa $ such that  $C\cap S \subseteq \{\alpha\in S\mid B\cap\alpha\in \mathcal A_\alpha \}$.
	\end{itemize}
\item $ \diamondsuit(S) $ asserts the existence of a sequence $ \langle A_\alpha \mid \alpha\in S \rangle $ such that:
	\begin{itemize}
		\item for all $ \alpha\in S $, $ A_\alpha \subseteq \alpha $;
		\item for every $B\s \kappa$, the set  $\{\alpha\in S\mid B\cap\alpha=A_\alpha \}$ is stationary.
	\end{itemize}
\item 	$ \clubsuit(S) $ asserts the existence of a sequence $ \langle A_\alpha\mid \alpha\in S \rangle $ such that:
	\begin{itemize}
		\item\label{Definiton clubsuit - Clause A_alpha} for all $ \alpha\in S\cap \acc(\kappa) $, $ A_\alpha $ is a cofinal subset of $\alpha$ of order type $\cf(\alpha)$;
		\item\label{Definiton clubsuit - Clause guess} for every cofinal subset $ B\subseteq \kappa$, the set  $\{\alpha\in S \mid A_\alpha \subseteq B \}$ is stationary.
	\end{itemize}	
\item \label{clubsuit_J(S)} 	$\clubsuit_{J}(S)$ asserts the existence of a matrix $\langle A_{\alpha,i}\mid\alpha\in S,~i<\cf(\alpha)\rangle$ such that:
	\begin{itemize}
		\item\label{clubsuits_J_ordertype_omega} For all $ \alpha\in S\cap\acc(\kappa) $,
		$ \langle A_{\alpha,i} \mid i<\cf(\alpha) \rangle $ is a sequence of pairwise disjoint cofinal subsets of $ \alpha$, each of order-type $\cf(\alpha)$;
		\item\label{clubsuits_J_unboundedsubset} For every cofinal subset $B\s\kappa$, the following set is stationary:
		$$\{\alpha\in S\mid\forall i<\cf(\alpha)[\sup(B\cap A_{\alpha,i})= \alpha]\}.$$
	\end{itemize} 
\end{enumerate}
\end{definition}
\begin{remark}\label{diamondsuit iff clubsuit and ch}\label{one cohen implies clubsuit_J} 
The principle $\diamondsuit^*$ was introduced by Kunen and Jensen in \cite{jensen1969some},
the principle $\diamondsuit$ was introduced by Jensen in \cite{Jensen_V=L_Diamond},
the principle $\clubsuit$ was introduced by Ostaszewski in \cite{Ostaszewski_space},
and the principle $\clubsuit_J$ was introduced by Juh\'{a}sz in \cite{Juhasz_clubsuit_ostaszewski} (under the name $(t)$).
It is not hard to see that for stationary $S'\s S\s\kappa$, $\diamondsuit^*(S')\implies\diamondsuit(S)\implies\clubsuit(S)\implies\clubsuit_J(S)$.
Devlin (see \cite[p.~507]{Ostaszewski_space}) proved that $\diamondsuit(S)\iff \clubsuit(S) + \kappa^{<\kappa}=\kappa$.
In \cite{Juhasz_clubsuit_ostaszewski}, Juh\'{a}sz 
proved that $\clubsuit_J(\omega_1)$ is adjoined by the forcing to add a Cohen real, and proved that the former suffices for the construction of an Ostaszewski space.
\end{remark}

To present our new guessing principle, we shall first need the following definition.

\begin{definition}\label{almost-disjoint ladder-system}  For a set of ordinals $S$:
\begin{enumerate}
\item A sequence $\langle A_\alpha\mid \alpha\in S\rangle$ 
is said to be an \emph{$\ad$-ladder system} iff the two hold:
	\begin{itemize}
		\item For all $\alpha\in S\cap\acc(\kappa)$, $A_\alpha$ is a cofinal subset of $\alpha$;
		\item For all two distinct $\alpha,\alpha'\in S$, $\sup(A_\alpha\cap A_{\alpha'})<\alpha$.
		\end{itemize}
\item A sequence $\langle \mathcal A_\alpha\mid \alpha\in S\rangle$ 
is said to be an \emph{$\ad$-multi-ladder system} iff the two hold:
	\begin{itemize}
		\item\label{clubsuit AD, Clause family softer} For all $\alpha\in S\cap\acc(\kappa)$, $\mathcal A_\alpha$ is a nonempty family consisting of pairwise disjoint cofinal subsets of $\alpha$;
		\item\label{clubsuit AD, Clause almost disjoint}  For all two distinct $A,A'\in\bigcup_{\alpha\in S}\mathcal A_\alpha$, $\sup(A\cap A')<\sup(A)$.
		\end{itemize}
\end{enumerate}
\end{definition}

Now, we are ready to present the new guessing principle.

\begin{definition}\label{clubsuit AD definition} For a family $\mathcal S$ of stationary subsets of $\kappa$, $\clubsuit_{\ad}(\mathcal S,{<}\theta)$
	asserts the existence of an $\ad$-multi-ladder system $\vec{\mathcal A}=\langle \mathcal A_{\alpha}\mid\alpha\in \bigcup\mathcal S\rangle$ such that:
	\begin{enumerate}
		\item\label{clubsuit AD, Clause family} For every $\alpha\in\bigcup\mathcal S$, $|\mathcal A_\alpha|=\cf(\alpha)$;
		\item\label{clubsuit AD, Clause cofinal set} For every $\mathcal B\s[\kappa]^{\kappa}$ with $|\mathcal B|<\theta$,
		and every $S\in\mathcal S$,
		the following set is stationary:
		$$G(S,\mathcal B):=\{\alpha\in S\mid  \forall (A,B)\in\mathcal A_\alpha\times\mathcal B~[\sup(A\cap B)= \alpha]\}.$$
	\end{enumerate} 	
\end{definition}

\begin{conv}\label{clubsuit AD simple definition}  We write $\clubsuit_{\ad}(\mathcal S,\theta)$ for $\clubsuit_{\ad}(\mathcal S,{<}(\theta+1))$,
and $\clubsuit_{\ad}(S)$ for $\clubsuit_{\ad}(\{S\},1)$.
\end{conv}

\begin{remark} For any $\chi\in\reg(\kappa)$ 
and any stationary $S\s E^\kappa_\chi$, $\clubsuit_J(S) \implies \clubsuit_{\ad}(S)$.
\end{remark}

\begin{figure}[H]
	\centering
	\begin{tikzcd}[row sep=scriptsize, column sep=scriptsize]
	 \diamondsuit(S) \arrow[d]\arrow[rr]{} &  {} & \diamondsuit(\omega_1)\arrow[dd] \\
	 \clubsuit(S) \arrow[d] & {} \text{Cohen real}\arrow[dl]\arrow[dr] &  {}  \\
	 \clubsuit_J(S) \arrow[d] & {} &  \exists\aleph_1\text{-Souslin Tree}\arrow[d]  \\
	\clubsuit_{\ad}(S) & {} & \forall\mathcal S~\clubsuit_{\ad}(\mathcal S,{<}\omega)\arrow[ll]   \\
	\end{tikzcd}
	\caption{Diagram of implications between the combinatorial principles under discussion, at the level of $\omega_1$.}
\end{figure}
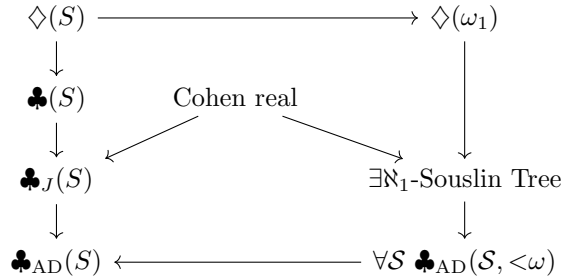

To motivate Definition~\ref{almost-disjoint ladder-system}, let us point out 
two easy facts concerning disjointifying $\ad$ systems.

\begin{prop}\label{Lemma - non-reflecting stat, diagonlization of initial seg}
	Suppose that $S$ is a non-reflecting stationary subset of $\kappa$, and $\vec A=\langle A_\alpha \mid \alpha\in S \rangle$ is an $\ad$-ladder system.
	Then there exists a sequence of functions $\langle f_\xi\mid \xi<\kappa\rangle$ such that, for every $\xi<\kappa$:
	\begin{enumerate}
	\item $f_\xi$ is a regressive function from $S\cap\xi$ to $\xi$;
	\item the sets in $\langle A_\alpha\setminus f_\xi(\alpha) \mid \alpha\in S\cap \xi\rangle$ are pairwise disjoint.
	\end{enumerate}
\end{prop}
\begin{proof} We recursively construct a sequence $\langle f_{\xi}\mid \xi<\kappa \rangle$ such that for every $\xi<\kappa$, Clauses (1) and (2) above hold true.
	
	$\br$ The cases where either $\xi=0$ or $\xi=\beta+1$ for $\beta\notin S$ are straightforward.

	$\br$ Suppose $\xi=\beta+1$ with $\beta\in S$ for which $f_\beta$ has already been defined. Define $g:S\cap\beta\rightarrow\beta$ via $g(\alpha):=\sup(A_\beta\cap A_\alpha)$.
	As $\vec A$ is an $\ad$-ladder system, $g$ is regressive. In effect, we may define a regressive function $f_\xi:S\cap\xi\rightarrow\xi$ via $f_\xi(\alpha):=\max\{f_\beta(\alpha),g(\alpha)\}$ for $\alpha\in S\cap \beta$,
	and $f_\xi(\beta):=0$. Notice that by the recursive assumption the function $f_\xi$ is as sought.
	
	$\br$ Suppose $\xi\in\acc(\kappa)$ for which $\langle f_\beta\mid\beta<\xi\rangle$ has already been defined. As $S$ is non-reflecting we may fix a club $C\subseteq \xi$ disjoint from $S$.
	For every $\alpha<\xi$, set $\alpha^-:=\sup(C\cap\alpha)$ and $\alpha^+:=\min(C\setminus \alpha)$. Note that, for every nonzero $\alpha\in S\cap\xi$, $\alpha^-<\alpha<\alpha^+$.
	Define a regressive function $f_\xi:S\cap \xi\rightarrow \xi$ via $f_\xi(\alpha):=\max\{f_{\alpha^+}(\alpha),\alpha^-\}$. Notice that by the recursive hypothesis, the function $f_\xi$ is as sought.
\end{proof}

\begin{prop}\label{Proposition - disjointify multi-ladder system} Suppose that  $S$ is a subset of $E^\kappa_{\ge\lambda}$ and $\vec{\mathcal A}=\langle \mathcal A_\alpha\mid \alpha\in S\rangle$ 
is an $\ad$-multi-ladder system. For any $\mathcal B\s\bigcup_{\alpha\in S}\mathcal A_\alpha$ with $|\mathcal B|\le\lambda$,
there exists a function $f:\mathcal B\rightarrow\kappa$ such that:
	\begin{enumerate}
	\item for every $B\in\mathcal B$, $f(B)\in B$;
	\item the sets in $\langle B\setminus f(B) \mid B\in\mathcal B\rangle$ are pairwise disjoint.
	\end{enumerate}
\end{prop}
\begin{proof} Given $\mathcal B\s\bigcup_{\alpha\in S}\mathcal A_\alpha$ with $|\mathcal B|\le\lambda$,
fix an injective enumeration $\langle B_\xi \mid \xi<|\mathcal B|\rangle$ of $\mathcal B$.
By the hypothesis on $\vec{\mathcal A}$, for every pair $\zeta<\xi<|\mathcal B|$, 
$$\sup(B_\xi\cap B_\zeta)<\sup(B_\xi)\in E^\kappa_{\ge\lambda}\s E^\kappa_{>\xi},$$
so that $\sup_{\zeta<\xi}\sup(B_\xi\cap B_\zeta)<\sup(B_\xi)$.

It follows that we may define a function $f:\mathcal B\rightarrow\kappa$ via:
$$f(B_\xi):=\begin{cases}\min(B_\xi),&\text{if }\xi=0;\\
\min\{\beta\in B_\xi\mid \sup\{\sup(B_\xi\cap B_\zeta)\mid \zeta<\xi\}<\beta\},&\text{otherwise}.
\end{cases}$$
Evidently, $f$ is as sought.
\end{proof}

Our next lemma shows, in particular, that for any $\chi\in\reg(\kappa)$
and any stationary $S\s E^\kappa_\chi$, 
$\clubsuit(S) \implies \clubsuit_{\ad}(\{S\},{<}\omega)$.
The reverse implication does not hold in general, as established by Corollary~\ref{clubadvsclub} below.
The proof of the lemma will make use of the following fact.

\begin{fact}[Brodsky-Rinot, {\cite[\S3]{paper23}}]\label{clubfacts}
For any stationary $S\s\kappa$, all of the following are equivalent:
	\begin{enumerate}
	\item $ \clubsuit(S) $;
	\item \label{many clubsuit from one} there exists a partition $ \langle S_i\mid i<\kappa \rangle $ of $ S $ into pairwise disjoint stationary sets such that	$ \clubsuit(S_i ) $ holds for each $ i<\kappa $;
\item \label{matrix clubsuit from one}
for any (possibly finite) cardinal $\theta$ such that $\kappa^\theta=\kappa$, there exists a matrix $ \langle A_{\alpha,\tau} \mid \alpha\in S,~\tau\le \theta \rangle $ such that, 
for every sequence $\langle A_\tau \mid \tau\leq\theta \rangle$ of cofinal subsets of $\kappa$, the following set is stationary in $\kappa$:
	 $$\{ \alpha\in S \mid \forall \tau\leq \theta~[A_{\alpha,\tau}\subseteq A_\tau\cap \alpha \ \& \ \sup(A_{\alpha,\tau})=\alpha] \} .$$
	\item \label{matrix clubsuit <omega hitting}
	there exists a sequence $\langle X_\alpha \mid \alpha\in S\rangle$ such that:
	\begin{itemize}
		\item for every $\alpha\in S\cap \acc(\kappa)$, $X_\alpha\subseteq [\alpha]^{<\omega} $ with $\mup(X_\alpha)=\alpha$;\footnote{Recall that $\mup$ was defined in Subsection~\ref{notationsubsection}.}
		\item\label{mupclub} for every $X\subseteq [\kappa]^{<\omega}$ with $\mup(X)=\kappa$, the following set is stationary: 
		$$\{ \alpha\in S \mid X_\alpha \subseteq X\}.$$
	\end{itemize}
	\end{enumerate}
\end{fact}

\begin{lemma}\label{lemma216} Suppose that $\clubsuit(S)$ holds for some stationary $S\s E^\kappa_\chi$ with $\chi\in\reg(\kappa)$.
	Then there exists a partition $\mathcal S$ of $S$ into $\kappa$ many stationary sets
	for which $\clubsuit_{\ad}(\mathcal S,{<}\omega)$ holds as witnessed by a sequence $\vec{\mathcal A}=\langle \mathcal A_\alpha\mid\alpha\in S\rangle$
	with $\otp(A)=\chi$ for all $A\in\bigcup_{\alpha\in S}\mathcal A_\alpha$.	
\end{lemma}
\begin{proof} By Fact~\ref{clubfacts}\eqref{many clubsuit from one}, fix a partition $ \langle S_i\mid i<\kappa \rangle $ of $S$ into pairwise disjoint stationary sets such that	$ \clubsuit(S_i ) $ holds for each $ i<\kappa $.
Set $\mathcal S:=\{ S_i\mid i<\kappa\}$.
Next, for each $i<\kappa$, let $\langle X_\alpha \mid \alpha\in S_i\rangle$ be a sequence as in Fact~\ref{clubfacts}\eqref{mupclub}.
Fix a surjection $h:\chi\rightarrow\chi$ such that $|h^{-1}\{j\}|=\chi$ for all $j<\chi$.

	To simplify the upcoming argument let us agree to write, for any two nonempty sets of ordinals $a,b$,
	``$a<b$'' iff $\alpha<\beta$ for all $(\alpha,\beta)\in a\times b$.

	Let $\alpha\in S\cap\acc(\kappa)$. 
	Recall that $X_\alpha\subseteq [\alpha]^{<\omega} $ and $\mup(X_\alpha)=\alpha$.
	Fix a strictly increasing sequence of ordinals $\langle \alpha_\zeta\mid \zeta<\chi\rangle$ that converges to $\alpha$.
	Now, by recursion on $\zeta<\chi$, we construct a sequence $\langle x_\zeta \mid \zeta<\chi\rangle$ such that, for every $\zeta<\chi$:
	\begin{enumerate}
		\item $ x_\zeta \in X_\alpha$, and
		\item\label{Clause inc. min} for every $\xi<\zeta$, $(x_\xi\cup\{\alpha_\zeta\})<x_\zeta$.
	\end{enumerate}
	Suppose $\zeta<\chi$ and that $\langle x_\xi \mid \xi<\zeta \rangle$ has already been defined.
	Evidently, $\eta:=\sup(\{ \max(x_\xi) \mid \xi<\zeta \}\cup\{\alpha_\zeta\})$ is $<\alpha$.
	So, as $\mup(X_\alpha)=\alpha$, we may let $x_\zeta:= x$ for some $x\in X_\alpha$ with $\min(x)>\eta$.
	
	This completes the construction. By Clause~(2), $\langle x_\zeta\mid \zeta<\chi\rangle$ is $<$-increasing,
	and 	$\mup\{ x_\zeta \mid \zeta<\chi \} = \alpha$.
	Finally, for every $j<\chi$, let $X^j_{\alpha}:=\{x_\zeta\mid \zeta<\chi, h(\zeta)=j\}$,
	and $A^j_\alpha:=\biguplus X^j_\alpha$.

	\begin{claim} $\langle A^j_\alpha\mid j<\chi\rangle$ is a sequence of pairwise disjoint cofinal subsets of $\alpha$, each of order-type $\chi$.
	\end{claim}
	\begin{proof} Let $j<\chi$.
	As $\mup\{ x_\zeta \mid \zeta<\chi, h(\zeta)=j \} = \alpha$,
	and as $\langle x_\zeta\mid \zeta<\chi, h(\zeta)=j \rangle$ is a $<$-increasing $\cf(\alpha)$-sequence of finite sets,
	we infer that $\sup(A^j_\alpha)=\alpha$,
	and that, for every $\beta<\alpha$,
	$\otp(A^j_\alpha\cap \beta)<\chi$. Altogether, $\otp(A^j_\alpha)=\chi$.
	
	Also, since $\langle x_\zeta\mid \zeta<\chi\rangle$ consists of pairwise disjoint sets,
	the elements of $\langle A^j_\alpha\mid j<\chi\rangle$ are pairwise disjoint.
	\end{proof}

	For every $\alpha\in S$, set $\mathcal A_\alpha:=\{ A_\alpha^j\mid j<\chi \}$.
	It immediately follows from the preceding claim that $\vec{\mathcal A}:=\langle \mathcal A_\alpha \mid \alpha \in S \rangle $ is an $\ad$-multi-ladder system.
	
	\begin{claim}For every finite $\mathcal B\s[\kappa]^{\kappa}$ and every $i<\kappa$,
		the following set is stationary:
		$$G(S_i,\mathcal B):=\{\alpha\in S_i\mid  \forall (A,B)\in\mathcal A_\alpha\times\mathcal B~[\sup(A\cap B)= \alpha]\}.$$
	\end{claim}
	\begin{proof}	
	Suppose that $\langle B_n\mid n<m\rangle$ is a finite sequence of cofinal subsets of $\kappa$.
	For each $n<m$, fix an injective enumeration $\langle b_{n,\iota} \mid \iota<\kappa \rangle$ of $B_n$.
	Set 	$X:=\{ \{ b_{n,\iota} \mid n<m\}  \mid \iota< \kappa \}$ and notice that $X\subseteq [\kappa]^{<\omega}$ with $\mup(X)=\kappa$.
	In effect, for every $i<\kappa$, the set $T_i:=\{ \alpha \in S_i \mid X_\alpha \subseteq X \}$ is stationary.
	Let $i<\kappa$. We claim that $T_i\s G(S_i,\mathcal B)$. To see this, let $\alpha\in T_i$ and $(A,B)\in\mathcal A_\alpha\times\mathcal B$ be arbitrary.
	Fix $j<\chi$ and $n<m$ such that $A=A^j_\alpha$ and $B=B_n$.
	
	As $X^j_{\alpha}\s X_\alpha\s X$, by the definition of $X$, for every $x\in X^j_\alpha$, $x\cap B_n\neq\emptyset$. As $\mup(X^j_\alpha)=\alpha$,
	it follows that $\sup(A_\alpha^j\cap B_n)=\alpha$.
	\end{proof}
	This completes the proof.
\end{proof}

We conclude this subsection by formulating a three-cardinal variant of $\clubsuit_{\ad}$:
\begin{definition}\label{threecardinalsvariant}
$\clubsuit_{\ad}(\mathcal S,\mu,{<}\theta)$ asserts the existence of a system 
$\vec{\mathcal A}=\langle \mathcal A_{\alpha}\mid\alpha\in \bigcup\mathcal S\rangle$ as in Definition~\ref{clubsuit AD definition},
but in which Clause~(1) is replaced by the requirement that, for every $\alpha\in\bigcup S$, $|\mathcal A_\alpha|=\mu$.
We write $\clubsuit_{\ad}(\mathcal S,\mu,\theta)$  for $\clubsuit_{\ad}(\mathcal S,\mu,{<}(\theta+1))$.
\end{definition}

It is clear that $\clubsuit_{\ad}(\mathcal S,\omega,{<}\theta)$ follows from $\clubsuit_{\ad}(\mathcal S,{<}\theta)$.
Also, the following lemma is obvious.
\begin{lemma}\label{adone}  For a family $\mathcal S$ of stationary subsets of $\kappa$, $\clubsuit_{\ad}(\mathcal S,1,2)$ holds
iff there exists an $\ad$-ladder system $\langle A_{\alpha}\mid\alpha\in \bigcup\mathcal S\rangle$ such that,
for all $B_0,B_1\in[\kappa]^{\kappa}$ and $S\in\mathcal S$,
the set $\{\alpha\in S\mid  \sup(A_\alpha\cap B_0)=\sup(A_\alpha\cap B_1)=\alpha\}$ is stationary.\qed
\end{lemma}

\subsection{Interlude on Souslin trees}\label{Guess Construction subsection}

Recall that a poset $\mathbf T=(T,{<_T})$ is a \emph{$\kappa$-Souslin tree} iff all of the following hold:
\begin{itemize}
\item $|T|=\kappa$;
\item $(T,{<_T})$ has no chains or antichains of size $\kappa$;
\item for every $x\in T$, $(x_\downarrow,<_T)$ is well-ordered.
\end{itemize}

For every $x\in T$, denote $\h(x):=\otp(x_\downarrow,<_T)$.
For every $A\s\kappa$, let $T\restriction A:=\{ x\in T\mid \h(x)\in A\}$.
Note that, for every $\alpha<\kappa$, $T_\alpha:=\{ x\in T\mid \h(x)=\alpha\}$ and $T\restriction\alpha$ have size ${<}\kappa$.

The next well-known lemma shows that Souslin trees are similar to Luzin spaces in the sense that every large subset of a Souslin tree is \emph{somewhere dense}.

\begin{lemma}[folklore]\label{Souslin_denseness}
	Suppose $\mathbf T=(T,{<_T})$ is a $\kappa$-Souslin tree and $B\subseteq T $ is a subset with $|B|=\kappa$. 
	Then there exists $w\in T$ such that $w^\uparrow\cap B$ is cofinal in $w^\uparrow$.
\end{lemma}
\begin{proof} Let $X$ denote the collection of all $x\in T$ such that $x^\uparrow\cap B$ is empty.
Let $A\s X$ be a maximal antichain in $X$.
As $\mathbf T$ is a $\kappa$-Souslin tree, $|A|<\kappa$, so we may find a large enough $\delta<\kappa$ such that $A\s T\restriction\delta$.
As $|B|=\kappa>|T\restriction(\delta+1)|$, let us fix $b\in B$ with $\h(b)>\delta$. 
Finally, let $w$ denote the unique element of $T_\delta$ with $w<_T b$.

\begin{claim} $w^\uparrow\cap B$ is cofinal in $w^\uparrow$.
\end{claim}
\begin{proof} Suppose not. Then there must exist some $x\in X$ with $w\le_T x$.
As $A$ is a maximal antichain in $X$, we may find $\bar x\in A$ which is comparable with $x$. 
As $\h(\bar x)<\delta\le\h(x)$, it follows that $\bar x<_T x$. As $\h(w)=\delta$ and $w\le_T x$, it  follows that $\bar x<_Tw\le_T x$.
In particular, $\bar x<_Tw<_Tb$, so that $b\in\bar x^\uparrow\cap B$, contradicting the fact that $\bar x\in X$.
\end{proof}
This completes the proof.
\end{proof}

\begin{definition} A $\kappa$-Souslin tree $\mathbf T=(T,{<_T})$ is said to be: 
\begin{itemize}
\item \emph{normal} iff for any $x\in T$ and $\alpha<\kappa$ with $\h(x)<\alpha$, there exists $y\in T_\alpha$ with $x<_T y$;
\item \emph{$\mu$-splitting} iff every node in $T$ admits at least $\mu$-many immediate successors,
that is, for every $x\in T$, $|\{ y\in T\mid x<_T y, \h(y)=\h(x)+1\}|\ge\mu$;
\item \emph{prolific} iff every $x\in T$ admits at least $\h(x)$-many immediate successors;
\item \emph{$\chi$-complete} iff any $<_T$-increasing sequence of elements from $T$, and of length ${<}\chi$, has an upper bound in $T$;
\item \emph{regressive} iff there exists a map $\rho:T\restriction\acc(\kappa)\rightarrow T$ satisfying the following:
\begin{itemize}
\item for every $x\in T\restriction\acc(\kappa)$, $\rho(x)<_T x$;
\item for all $\alpha\in\acc(\kappa)$ and $x,y\in T_\alpha$, 
if $\rho(x)<_T y$ and $\rho(y)<_T x$, then $x=y$;
\end{itemize}
\item \emph{ordinal-based} iff $T=\kappa$ and, for all $x,y\in T$, if $\h(x)<\h(y)$, then $x\in y$.
\end{itemize}

A subset $B\s T$ is said to be an \emph{$\alpha$-branch} iff $(B,<_T)$ is linearly ordered and $\{\h(x)\mid x\in B\}=\alpha$;
it is said to be \emph{vanishing} iff it has no upper bound in $T$.
\end{definition}

\begin{definition} A $\lambda^+$-Souslin tree is said to be \emph{maximally-complete} iff it is $\chi$-complete
for $\chi:=\log_\lambda(\lambda^+)$.
\end{definition}

Note that the existence of a $\cf(\lambda)$-complete $\lambda^+$-Souslin tree
is equivalent to the conjunction of ``$\lambda^{<\cf(\lambda)}=\lambda$'' and ``there is a maximally-complete $\lambda^+$-Souslin tree''.

\begin{prop}[folklore]\label{ordinalsbased} 	For cardinals $\chi,\mu<\cf(\kappa)=\kappa$, if there exists a $\kappa$-Souslin tree which is $\chi$-complete (resp.~regressive), 
then there exists an ordinal-based $\mu$-splitting, normal, prolific 
$\kappa$-Souslin tree which is $\chi$-complete (resp.~regressive).
\end{prop}
\begin{proof} Suppose $\mathbf T=(T,{<_T})$ is a $\kappa$-Souslin tree. 
By a standard fact (see \cite[Lemma~2.4]{rinot20}), we may fix a club $E\subseteq \kappa$ such that $(T\restriction E, <_T)$ is normal and splitting.
Consider the set $D:=\{\alpha<\kappa\mid \otp(E\cap\alpha)=\mu^\alpha\}$ which is a subclub of $E$.
It is clear that $\mathbf T':=(T\restriction D, <_T)$ is a normal $\kappa$-Souslin tree. 
\begin{claim} \begin{enumerate}
\item $\mathbf T'$ is prolific and 
$\mu$-splitting;
\item if $\mathbf T$ is $\chi$-complete, then so is $\mathbf T'$;
\item if $\mathbf T$ is regressive, then so is $\mathbf T'$.
\end{enumerate}
\end{claim}
\begin{proof} (1) Fix an arbitrary node $x$ of $\mathbf T'$, so that $x\in T\restriction D$.
Let $\delta:=\min(D\setminus(\h_{\mathbf T}(x)+1))$.
As $\delta\in D\s E$ and $(T\restriction E, <_T)$ is normal, let us fix $z\in T_\delta$ with $x<_T z$.
Let $e:=\{\varepsilon\in E\mid \h(x)<\varepsilon<\delta\}$.
Note that from $\otp(E\cap\delta)=\mu^\delta$, it follows that $\otp(e)=\delta$ and $|e|\ge\mu$.

For every $\varepsilon\in e$, let $y_\varepsilon$ denote the unique element of $T_\varepsilon$ satisfying $y_\varepsilon <_T z$,
and denote $\varepsilon^+:=\min(e\setminus(\varepsilon+1))$.
Then, using the fact that $(T\restriction E, <_T)$ is normal and splitting, 
for every $\varepsilon\in e$, pick  $\hat y_\varepsilon\in T_{\varepsilon^+}$ such that $y_\varepsilon<_T \hat y_\varepsilon$ and $\hat y_\varepsilon\neq y_{\varepsilon^+}$,
and then pick $z_\varepsilon\in T_\delta$ with $\hat y_\varepsilon<_T z_\varepsilon$.
Then $\{ z_\varepsilon\mid \varepsilon \in e\}$ consists of $|e|$-many immediate successors of $x$ in $\mathbf T'$.

(2) Since $D$ is closed.

(3) Suppose $\rho:T\restriction\acc(\kappa)\rightarrow T$ witnesses that $\mathbf T$ is regressive. Define $\rho':T\restriction\acc(D)\rightarrow T\restriction D$ as follows.
Given $\alpha\in\acc(D)$ and $x\in T_\alpha$, let $\delta:=\min(D\setminus(\h_{\mathbf T}(\rho(x))+1))$,
and then let $\rho'(x)$ be the unique $y<_T x$ with $\h_{\mathbf T}(y)=\delta$.
It is clear that $\rho'$ witnesses that $\mathbf T'$ is regressive.
\end{proof}

Recursively define a sequence of injections $\langle \pi_\alpha:T_\alpha\rightarrow\kappa\mid \alpha\in D\rangle$ such that for, every $\alpha\in D$:
\begin{itemize}
\item For every $\alpha'\in D\cap\alpha$, $\im(\pi_{\alpha'})\cap\im(\pi_\alpha)=\emptyset$;
\item $\biguplus\{\im(\pi_{\alpha'})\mid \alpha'\in D\cap(\alpha+1)\}$ is an ordinal.
\end{itemize}
Evidently, $\pi:=\bigcup_{\alpha\in D}\pi_\alpha$ is an injection from $T\restriction D$ onto $\kappa$.
Let ${\lhd}:=\{ (\pi(x),\pi(y))\mid (x,y)\in{<_T}\}$. Then $(\kappa,{\lhd})$ is an ordinal-based $\kappa$-Souslin tree order-isomorphic to $\mathbf T'$.
	\end{proof}

A richer introduction to abstract transfinite trees, Aronszajn trees and Souslin trees may be found in Section~2 of \cite{paper23},
and a comprehensive treatment of the consistency of existence of Souslin trees may be found in Section~6 of the same paper.
For our purpose, it suffices to mention the following fact:
\begin{fact}\label{fact216} For an infinite cardinal $\lambda$ satisfying $\diamondsuit(\lambda^+)$:\footnote{Note that, by \cite{Sh:922}, for any \emph{uncountable} cardinal $\lambda$, $\diamondsuit(\lambda^+)$ holds iff $2^\lambda=\lambda^+$.}
\begin{enumerate}
\item \cite{Jensen_V=L_Diamond} Assuming $\lambda^{<\lambda}=\lambda$, if $\diamondsuit(E^{\lambda^+}_\lambda)$ holds, then there exists a $\lambda$-complete $\lambda^+$-Souslin tree;
\item \cite{paper37} Assuming $\lambda^{<\lambda}=\lambda$, if $\square(\lambda^+,{<}\lambda)$ holds, then there exists a $\lambda$-complete $\lambda^{+}$-Souslin tree;
\item \cite{paper24,paper51} Assuming $\lambda^{\aleph_0}=\lambda$ or $\lambda\ge\beth_\omega$ or $\mathfrak b\le\lambda<\aleph_\omega$, 
if $\square(\lambda^+)$ holds, then there exists a maximally-complete $\lambda^+$-Souslin tree and there exists a regressive $\lambda^+$-Souslin tree;
\item\cite{paper31} If $\square_{\lambda^+}$ holds, then there exists a $\lambda^+$-complete $\lambda^{++}$-Souslin tree.
\end{enumerate}
\end{fact}

We now introduce a new characteristic of Souslin trees.

\begin{definition}[The levels of vanishing branches] For a $\kappa$-Souslin tree $\mathbf T=(T,{<_T})$, let $V(\mathbf T)$ denote the set of all $\alpha\in\acc(\kappa)$
such that, for every $x\in T\restriction\alpha$, there exists a vanishing $\alpha$-branch containing $x$.
\end{definition}

It follows from a theorem of Shelah \cite{Sh:624} that it is consistent that for some Mahlo cardinal $\kappa$, there exists a $\kappa$-Souslin tree $\mathbf T$ for which $V(\mathbf T)=\emptyset$.

\begin{lemma}\label{regressivecharacteristic} Suppose that $\mathbf T=(T,{<_T})$ is a normal $2$-splitting regressive $\kappa$-Souslin tree.
Then $V(\mathbf T)\supseteq E^\kappa_\omega$.
\end{lemma}
\begin{proof} Towards a contradiction, suppose that $\alpha\in E^\kappa_\omega\setminus V(\mathbf T)$. 
Fix $x\in T\restriction\alpha$ such that every $\alpha$-branch $B$ with $x\in B$ has an upper bound in $T$.
Fix a strictly increasing sequence of ordinals $\langle \alpha_n\mid n<\omega\rangle$ that converges to $\alpha$, and $\alpha_0:=\h(x)$.
We shall recursively construct an array $\langle x_t\mid t\in{}^{<\omega}2\rangle$ in such a way that $x_t\in T_{\alpha_{|t|}}$.
Set $x_\emptyset:=x$. Now, for every $t\in{}^{<\omega}2$ such that $x_t$ has already been defined, since $\mathbf T$ is $2$-splitting and normal,
we may find $y\neq z$ in $T_{\alpha_{|t|+1}}$ with $x<_T y,z$; then, let $x_{t{}^\smallfrown\langle 0\rangle}:=y$ and  $x_{t{}^\smallfrown\langle 1\rangle}:=z$.
Next, given $t\in{}^{\omega}2$, let $B_t:=\{ y\in T\restriction\alpha\mid \exists n<\omega(y<_T x_{t\restriction n})\}$.
As $B_t$ is an $\alpha$-branch containing $x$, it must have a bound $b_t\in T$.
Clearly, $\h(b_t)\ge\alpha$, and we may moreover assume that $\h(b_t)=\alpha$.
Note that the construction secures that, for all $t\neq t'$ in ${}^\omega2$, $b_t\neq b_{t'}$.

Let $\rho:T\restriction\acc(\kappa)\rightarrow T$ be a witness to the fact that $\mathbf T$ is regressive.
Next, for every $t\in{}^\omega2$, fix a large enough $n_t<\omega$ such that $\rho(b_t)<_T x_{t\restriction n_t}$.
By the pigeonhole principle, we may now fix $s\in{}^{<\omega}2$ such that $\{t\in{}^\omega2\mid t\restriction n_t=s\}$ is uncountable.
Pick $t\neq t'$ in ${}^\omega2$ such that $t\restriction n_{t}=s=t'\restriction n_{t'}$.
Then, $\rho(b_t)<_T x_s<_Tb_{t'}$ and $\rho(b_{t'})<_T x_s<_T b_t$, 
contradicting the fact that $b_t\neq b_{t'}$.
\end{proof}
\begin{remark} It follows that if $\mathbf T$ is a normal $2$-splitting coherent $\kappa$-Souslin tree, then $V(\mathbf T)=E^\kappa_\omega$.
A consistent construction of such a tree may be found in \cite[Proposition~2.5]{paper22}.
\end{remark}

\begin{lemma}\label{vanishcomplete} 
Suppose that $\mathbf T=(T,{<_T})$ is a normal maximally-complete $\lambda$-splitting $\lambda^+$-Souslin tree.
Then $V(\mathbf T)\supseteq E^{\lambda^+}_{\chi}$ for the (regular) cardinal $\chi:=\log_\lambda(\lambda^+)$.
\end{lemma}
\begin{proof} Towards a contradiction, suppose that $\alpha\in E^{\lambda^+}_{\chi}\setminus V(\mathbf T)$. 
Fix $x\in T\restriction\alpha$ such that every $\alpha$-branch $B$ with $x\in B$ has an upper bound in $T$.
Fix a strictly increasing and continuous sequence of ordinals $\langle \alpha_\epsilon\mid \epsilon<\chi\rangle$ that converges to $\alpha$, and $\alpha_0:=\h(x)$.
Very much like the proof of Lemma~\ref{regressivecharacteristic}, we may recursively construct an array $\langle x_t\mid t\in{}^{<\chi}\lambda\rangle$ in such a way that:
\begin{itemize}
\item $x_\emptyset=x$;
\item for all $t\in{}^{<\chi}\lambda$, $x_t\in T_{\alpha_{\dom(t)}}$;
\item for all $t,s\in {}^{<\chi}\lambda$, if $t\s s$, then $x_t<_T x_s$;
\item for all $t\in{}^{<\chi}\lambda$ and $i<j<\lambda$, $x_{t{}^\smallfrown\langle i\rangle}\neq x_{t{}^\smallfrown\langle j\rangle}$.
\end{itemize}

For each $t\in{}^{\chi}\lambda$, find $b_t\in T_\alpha$ such that, for every $\epsilon<\chi$, $x_{t\restriction\epsilon}<_T b_t$.
Then, $\{ b_t\mid t\in{}^{\chi}\lambda\}$ is an antichain of size $\ge\lambda^+$ in $\mathbf T$. This is a contradiction.
\end{proof}
\begin{remark} It follows that if $\mathbf T$ is a normal $\lambda$-complete $\lambda$-splitting $\lambda^+$-Souslin tree, then 
$V(\mathbf T)=E^{\lambda^+}_{\lambda}$.
\end{remark}

\subsection{Deriving our guessing principle from a Souslin tree}
\begin{theorem}\label{thm111}  
Suppose that $\mathbf T=(T,{<_T})$ is an ordinal-based $\chi$-splitting $\kappa$-Souslin tree,
and $\theta\le\chi$ is a cardinal satisfying $\kappa^{<\theta}=\kappa$.
Then, for every collection $\mathcal S$ of pairwise disjoint stationary subsets of $V(\mathbf T)\cap E^\kappa_{\chi}$,
there exists an $\ad$-multi-ladder system $\langle \mathcal A_{\alpha}\mid\alpha\in E^\kappa_{\chi}\rangle$ 
satisfying the following.
For every $\mathcal B\s[\kappa]^{\kappa}$ with $|\mathcal B|<\theta$,
		every $S\in\mathcal S$,
		and every cardinal $\mu<\kappa$,
		the following set is stationary:
		$$G_{\ge\mu}(S,\mathcal B):=\{\alpha\in S\mid |\mathcal A_\alpha|\ge\mu\ \&\ \forall (A,B)\in\mathcal A_\alpha\times\mathcal B~[\sup(A\cap B)= \alpha]\}.$$
\end{theorem}
\begin{proof} 
	As $\mathbf T$ is $\chi$-splitting and prolific, for each $w\in T$, we may fix an injective sequence $\langle w_i\mid i<\max\{\chi,\h(w)\}\rangle$ consisting of immediate successors of $w$.

As $\kappa^{<\theta}=\kappa$, we may fix an injective enumeration $\langle W_\eta\mid \eta<\kappa\rangle$ of all subsets $W$ of $T$ such that:
\begin{itemize}
\item $0<|W|<\theta$, and
\item $\h\restriction W$ is a constant function whose sole value is in $[\chi,\kappa)$. 
\end{itemize}

	\begin{claim}\label{claim1112} For every $\alpha\in E^\kappa_{\chi}$,
		there exists a cofinal subset $A_\alpha$ of $\alpha$ which is an antichain in $\mathbf T$.
	\end{claim}
	\begin{proof}  Fix an arbitrary $\alpha\in E^\kappa_{\chi}$,
		and let $X$ be an arbitrary cofinal subset of $\alpha$ of order-type $\chi$.
		
		If there exists some $\epsilon<\kappa$ such that $|X\cap T_\epsilon|=\chi$,
		then we are done by letting $A_\alpha:=X\cap T_\epsilon$. Thus, hereafter assume this is not the case,
		and pick a sequence $\langle x^j\mid j<\chi\rangle$
		of elements of $X$ for which $\langle \h(x^j)\mid j<\chi\rangle$ is strictly increasing.
		For notational simplicity, let us assume that $\{ x^j\mid j<\chi\}=X$.
		Now, there are two cases to consider:
		
		$\br$ Suppose that there exists a node $w\in T$ such that, for every $i<\chi$, there exists $j_i<\chi$ such that $x^{j_i}$ extends $w_i$.
		Recalling that $\langle w_i\mid i<\chi\rangle$ is an injective sequence of immediate successors of $w$,
		we infer that $A_\alpha:=\{ x^{j_i}\mid i<\chi\}$ is an antichain as sought.
		
		$\br$ Suppose not. In particular, for each $j<\chi$ (using $w:=x^j$), we may fix $i_j<\chi$
		such that $(x^j)_{i_j}$ is not extended by any element of $X$.
		We claim that $A_\alpha:=\{ (x^j)_{i_j}\mid j<\chi\}$ is as sought.
		
		For every $j<\chi$, we have $\h(x^j)<\h((x^j)_{i_j})<\h(x^{j+2})$,
		so that $x^j\in (x^j)_{i_j}\in x^{j+2}$, and hence $A_\alpha$ is yet another cofinal subset of $\alpha$.
		To see that $A_\alpha$ is an antichain, suppose that there exists a pair $j<j'$ 
		such that $(x^j)_{i_j}$ is comparable with $(x^{j'})_{i_{j'}}$.
		As $\h((x^{j})_{i_j})=\h(x^j)+1<\h(x^{j'})+1=\h((x^{j'})_{i_{j'}})$,
		it follows that $(x^j)_{i_j}$ is extended by $(x^{j'})_{i_{j'}}$,
		and in particular, $(x^j)_{i_j}$ is extended by $x^{j'}$ which is an element of $X$, contradicting the choice of $i_j$.
	\end{proof}
	
	Next, suppose that we are given a collection $\mathcal S$ of pairwise disjoint stationary subsets of $V(\mathbf T)\cap E^\kappa_{\chi}$.
	As $T=\kappa$ and $|T\restriction\alpha|<\kappa$ for all $\alpha<\kappa$, $C:=\{ \alpha<\kappa\mid \alpha=T\restriction\alpha\}$ is a club in $\kappa$.
	Let $\langle S_\eta\mid \eta<\kappa\rangle$ be a sequence of pairwise disjoint subsets of $E^{\kappa}_{\chi}\cap\acc(C)$ satisfying:
	\begin{itemize}
		\item  For every $S\in\mathcal S$, $S_\eta\cap S$ is stationary;
		\item  For every $\eta<\kappa$, $\min(S_\eta)$ is greater than the unique element of $\{\h(w)\mid w\in W_\eta\}$,
		which we hereafter denote by $\epsilon_\eta$. 
	\end{itemize}
	Let $R:=E^\kappa_{\chi}\setminus\biguplus_{\eta<\kappa}S_\eta$.
	
	For every $\alpha\in R$,  we appeal to Claim~\ref{claim1112} and 
	pick a cofinal subset $A_\alpha\s\alpha$ which is an antichain in $\mathbf T$. 
	By possibly thinning out, we may also assume that $\otp(A_\alpha)=\cf(\alpha)$.
	Then, we set $\mathcal A_\alpha:=\{A_\alpha\}$.
	
	Next, let $\alpha\in \biguplus_{\eta<\kappa}S_\eta$ be arbitrary.
	Let $\eta<\kappa$ be the unique ordinal such that $\alpha\in S_\eta$.
	As $\alpha\in V(\mathbf T)$,
	for every $w\in W_\eta$ and every $i<\epsilon_\eta$, we may find a subset $A^w_{\alpha,i}$ of $T$ such that:
	\begin{enumerate}
		\item $A^w_{\alpha,i}$ is a chain with minimal element $w_i$;
		\item $\{ \h(x)\mid x\in A^w_{\alpha,i}\}=\alpha\setminus\epsilon_\eta+1$;
		\item there exists no $z\in T_\alpha$ such that, for all $x\in A^w_{\alpha,i}$, $x<_T z$.
	\end{enumerate}
	
	Finally, let $\mathcal A_\alpha:=\{ A_{\alpha,i}\mid i<\epsilon_\eta\}$, where $A_{\alpha,i}:=\bigcup_{w\in W_\eta}A^w_{\alpha,i}$.
	
	\begin{claim}\label{Claim - Guess disjoint cofinal}
		$ \langle A_{\alpha,i} \mid i<\epsilon_\eta\rangle $ is a sequence of pairwise disjoint cofinal subsets of $\alpha$.
	\end{claim}
	\begin{proof} Let $w,u\in W_\eta$ and $i,j<\epsilon_\eta$ and suppose that $x\in A^w_{\alpha,i}\cap A^{u}_{\alpha,j}$.
		By Clause~(1) above, $x$ extends both $w_i$ and $u_{j}$.
		But $w_i$ and $u_{j}$ are predecessors of $x$ at the same level $T_{\epsilon_\eta+1}$, so that $w_i=u_j$ and it easily follows that $(w,i)=(u,j)$.
		
		For all $w\in W_\eta$ and $i<\epsilon_\eta$, it follows from $\alpha\in C$ and Clause~(2) above that $A^w_{\alpha,i}\s\alpha$,
		so that $A_{\alpha,i}\s\alpha$.
		In addition, as $\alpha\in\acc(C)$, if we pick any $w\in W_\eta$, then it follows from Clause~(2) above that $\sup\{ \beta\in C\mid A^w_{\alpha,i}\cap T_\beta\neq\emptyset\}=\alpha$,
		and hence $A_{\alpha,i}$ is cofinal in $\alpha$.
	\end{proof}
	
	\begin{claim}\label{Claim - Guess almost disjoint} Let $A,A'\in\bigcup_{\alpha\in E^{\kappa}_{\chi}}\mathcal A_\alpha$ with $A\neq A'$. Then $\sup(A\cap A')<\sup(A)$.
	\end{claim}
	\begin{proof} Let $\alpha,\alpha'$ be such that $A\in\mathcal A_\alpha$ and $A'\in\mathcal A_{\alpha'}$.
		As the elements of $\mathcal A_\alpha$ are pairwise disjoint, we may assume that $\alpha\neq\alpha'$.
		If $\alpha'<\alpha$, then $\sup(A\cap A')\le\alpha'<\alpha$, so assume that $\alpha'>\alpha$.
		
		$\br$ If $\alpha'\in R$, then $\otp(A\cap A')\le\otp(\alpha\cap A')<\otp(A')=\cf(\alpha)$, so that $\sup(A\cap A')<\alpha$.
		
		$\br$ If $\alpha'\notin R$ and $\alpha\in R$,
		then $A$ is antichain, while $A'$ is the union of ${<}\theta$ many chains,
		so that $|A\cap A'|<\theta\le\chi=\cf(\alpha)$,
		and again $\sup(A\cap A')<\alpha$.
		
		$\br$ If $\alpha,\alpha'\notin R$, then let $\eta,\zeta,i,j$ be such that $A=\bigcup_{w\in W_\eta}A^w_{\alpha,i}$
		and  $A'=\bigcup_{u\in W_\zeta}A^u_{\alpha,j}$.
		Since $\max\{|W_\eta|,|W_\zeta|\}<\theta\le\chi=\cf(\alpha)$, it suffices to show that for each $w\in W_\eta$ and $u\in W_\zeta$,
		$\sup(A^w_{\alpha,i}\cap A^u_{\beta,j})<\alpha$.
		But the latter follows from Clause~(3) above together with the fact that $\alpha\in C$.
	\end{proof}
 	
	Thus, we are left with verifying the following.
	
	\begin{claim}\label{Claim - Guess hitting} 
		Suppose $\mathcal B=\{B_\tau\mid \tau<\theta'\}$ is a family of cofinal subsets of $\kappa$, with $\theta'<\theta$.
		Suppose $S\in\mathcal S$ and $\mu<\kappa$ is some cardinal.
		Then the set $G_{\ge\mu}(S,\mathcal B)$ is stationary.
	\end{claim}
	\begin{proof} Recalling Lemma~\ref{Souslin_denseness}, for each $\tau<\theta'$, we may fix $w^\tau\in T$ such that $(w^\tau)^\uparrow\cap B_\tau$ is cofinal in $(w^\tau)^\uparrow$.
		Since $\mathbf T$ is normal, we may extend the said elements to ensure that $\h\restriction\{ w^\tau\mid \tau<\theta'\}$ is a constant function 
		whose sole value is some $\epsilon<\kappa$ with $\epsilon\ge\max\{\chi,\mu\}$.
		It follows that there exists (a unique) $\eta<\kappa$ such that $W_\eta=\{ w^\tau\mid \tau<\theta'\}$.
		Next, since $\mathbf T$ has no antichains of size $\kappa$, we may fix a sparse enough club $C'\s C$ with $\min(C')>\epsilon_\eta$ such that,
		for every pair of ordinals $\gamma<\beta$ from $C'$ and every $w^\tau\in W_\eta$,
		the set $B_\tau\cap((w^\tau)^\uparrow)\setminus(T\restriction\gamma)$ contains a maximal antichain in itself which is a subset of $T\restriction\beta$.
		
		Consider the stationary set $\Gamma:=S_\eta\cap\acc(C')$.
		Now, let $\alpha\in\Gamma$ be arbitrary. By Claim~\ref{Claim - Guess disjoint cofinal}, $|\mathcal A_\alpha|=|\epsilon_\eta|\ge\mu$.
		Next, let $\tau<\theta'$ and $A\in\mathcal A_\alpha$ be arbitrary.
		Find $i<\epsilon_\eta$ such that $A=A_{\alpha,i}=\bigcup_{w\in W_\eta}A^w_{\alpha,i}$. In particular, $A\supseteq A^{w_\tau}_{\alpha,i}$.
		As $\sup(C'\cap\alpha)=\alpha$, it thus suffices to show that for every $\gamma\in C'\cap\alpha$,  $\sup(A^{w^\tau}_{\alpha,i}\cap B_\tau)\ge\gamma$.
		For this, let $\gamma\in C'\cap\alpha$ be arbitrary. Let $\beta:=\min(C'\setminus(\gamma+1))$.
		Let $x$ denote the unique element of $A^{w^\tau}_{\alpha,i}\cap T_\beta$. As $w^\tau<_T (w^\tau)_i\le_T x$, we may find $a\in B_\tau$ with $x\le_T a$.
		As $\gamma<\beta$ is a pair of elements of $C'$, it follows that there exists $a'\in B_\tau$ with $a'\le_T a$ such that $\gamma\le \h(a')<\beta$.
		As $x,a'\le_T a$ and $\h(a')>\epsilon_\eta=\h(w^\tau)$, it follows that $a'\in A^{w^\tau}_{\alpha,i}\setminus (T\restriction\gamma)$, as sought.
	\end{proof}
	This completes the proof of Theorem~\ref{thm111}.
\end{proof}

\begin{cor}\label{Cor - regressive Souslin implies almost disjoint clubsuit omega} Suppose that $\mathbf T$ is a $\kappa$-Souslin tree. For every $\chi\in\reg(\kappa)$ and 
	every collection $\mathcal S$ of pairwise disjoint stationary subsets of $E^{\kappa}_{\chi}$,
	any of the following implies that  $\clubsuit_{\ad}(\mathcal S, {<}\chi)$ holds:
	\begin{enumerate}
	\item $\kappa=\lambda^+$, $\lambda^{\aleph_0}>\lambda$ and $\chi=\aleph_0$;
	\item $\kappa=\lambda^+$, $\chi=\log_\lambda(\lambda^+)$ and $\mathbf T$ is maximally-complete;
	\item $\chi=\aleph_0$ and $\mathbf T$ is regressive.
	\end{enumerate}
\end{cor}
\begin{proof} (1) This follows from Clause~(2).

(2) Appeal to Proposition~\ref{ordinalsbased}, 
Lemma~\ref{vanishcomplete} and Theorem~\ref{thm111} using $(\kappa,\mu,\theta):=(\lambda^+,\chi,\chi)$
to get a sequence $\langle \mathcal A_{\alpha}\mid\alpha\in E^{\lambda^+}_{\chi}\rangle$.
The only thing that possibly does not fit is that there may be $\alpha\in\bigcup\mathcal S$ for which $|\mathcal A_\alpha|\neq\chi$.
But this is easy to fix:
\begin{itemize}
\item[$\br$] For any $\alpha\in \bigcup\mathcal S$ such that $|\mathcal A_\alpha|>\chi$,
replace $\mathcal A_\alpha$ by some subset of it of size $\chi$.

\item[$\br$] For any $\alpha\in \bigcup\mathcal S$ such that $|\mathcal A_\alpha|<\chi$,
pick $A\in\mathcal A_\alpha$ and replace $\mathcal A_\alpha$ by some partition of $A$ into $\chi$ many sets.
\end{itemize}
(3) Appeal to Proposition~\ref{ordinalsbased}, 
	Lemma~\ref{regressivecharacteristic} and Theorem~\ref{thm111} using $(\mu,\theta):=(\omega,\omega)$
	to get a sequence $\langle \mathcal A_{\alpha}\mid\alpha\in E^{\kappa}_{\omega}\rangle$.
	The only thing that possibly does not fit is that there may be $\alpha\in\bigcup\mathcal S$ for which $|\mathcal A_\alpha|\neq\omega$.
	But we may fix it as in the previous case.
\end{proof}

Theorem~A now follows immediately.

\begin{cor}\label{Souslin implies almost disjoint clubsuit}  \begin{enumerate}
\item If there exists a $\cf(\lambda)$-complete $\lambda^+$-Souslin tree,
then for every partition $\mathcal S$ of $E^{\lambda^+}_{\cf(\lambda)}$ into stationary sets,
$\clubsuit_{\ad}(\mathcal S,{<}{\cf(\lambda)})$ holds.
\item If there exists a regressive $\kappa$-Souslin tree,
then for every partition $\mathcal S$ of $E^\kappa_\omega$ into stationary sets,
$\clubsuit_{\ad}(\mathcal S,{<}\omega)$ holds.\qed
\end{enumerate}
\end{cor}

\begin{cor}\label{w1souslin} If there exists an $\omega_1$-Souslin tree, then $\clubsuit_{\ad}(\{\omega_1\},{<}\omega)$ holds.\qed
\end{cor}

\begin{cor}\label{clubadvsclub} It is consistent with $\ch$ that $\clubsuit_{\ad}(\{\omega_1\},{<}\omega)$ holds, but $\clubsuit(\omega_1)$ fails.
\end{cor}
\begin{proof} Start with a model of $\gch+\neg\diamondsuit(\omega_1)$ (e.g., Jensen's model \cite{Jensen_The_Souslin_Problem} of $\gch$ with no $\aleph_1$-Souslin trees).
Now, force to add a single Cohen real and work in the corresponding extension.
As this is a countable forcing, $\ch$ still holds and $\diamondsuit(\omega_1)$ still fails,
so that by Devlin's theorem (see Remark~\ref{diamondsuit iff clubsuit and ch}), $\clubsuit(\omega_1)$ fails as well.
Finally, by a theorem of Shelah \cite{MR768264}, the forcing to add a Cohen real introduces an $\omega_1$-Souslin tree,
so that, by the preceding corollary, $\clubsuit_{\ad}(\{\omega_1\},{<}\omega)$ holds.
\end{proof}

\begin{cor} The assertion that $\clubsuit_{\ad}(\mathcal S,{<}\omega)$ holds
for every partition $\mathcal S$ of $\omega_1$ into stationary sets is consistent with any of the two:
\begin{enumerate}
\item $\clubsuit_J(\omega_1)$ fails;
\item $\clubsuit_J(S)$ fails for some stationary $S\s\omega_1$, and $\ch$ holds.
\end{enumerate}
\end{cor}
\begin{proof} For a stationary subset $S$ of $\omega_1$, let $\unif(S,2)$ assert
that for every sequence of functions $\vec f=\langle f_\alpha:A_\alpha\rightarrow2\mid \alpha\in \acc(\omega_1)\cap S\rangle$ where 
each $A_\alpha$ is a cofinal subset of $\alpha$ of order-type $\omega$,
there exists a function $f:\omega_1\rightarrow2$ that uniformizes $\vec f$, i.e., for every $\alpha\in S\cap\acc(\omega_1)$, $\Delta(f,f_\alpha):=\{ \beta\in A_\alpha\mid f(\beta)\neq f_\alpha(\beta)\}$ is finite.
\begin{claim} $\clubsuit_J(S)$ refutes $\unif(S,2)$.
\end{claim}
\begin{proof} Suppose that $\langle A_{\alpha,i}\mid\alpha\in S,~i<\omega\rangle$ is as in Definition~\ref{principles}(4).
In particular, for every $\alpha\in S\cap\acc(\omega_1)$, $A_\alpha:= A_{\alpha,0}\uplus A_{\alpha,1}$ is a cofinal subset of $\alpha$ of order-type $\omega$,
and we may define a function $f_\alpha:A_\alpha\rightarrow2$ via $f_\alpha(\beta):=0$ iff $\beta\in A_{\alpha,0}$.
Towards a contradiction, suppose that there exists a function $f:\omega_1\rightarrow 2$ that uniformizes $\vec f:=\langle f_\alpha\mid \alpha\in S\cap\acc(\omega_1)\rangle$.
By the pigeonhole principle, pick $j<2$ for which $B:=\{\beta<\omega_1\mid f(\beta)=j\}$ is uncountable.
Now, fix $\alpha\in S\cap\acc(\omega_1)$ such that $\sup(B\cap A_{\alpha,i})=\alpha$ for all $i<\omega$.
Let $i:=1-j$. Pick $\beta\in B\cap A_{\alpha,i}\setminus \Delta(f,f_\alpha)$.
Then $j=f(\beta)=f_\alpha(\beta)=i$. This is a contradiction.
\end{proof}

(1) As made clear by the proof of \cite[Theorem~5.2]{DvSh:65}, 
it is possible to force $\unif(\omega_1,2)$ via a finite support iteration of Knaster posets. In particular, 
if we start with a ground model with an $\aleph_1$-Souslin tree, then we can force 
$\unif(\omega_1,2)$ without killing the tree. 
Now appeal to Corollary~\ref{Cor - regressive Souslin implies almost disjoint clubsuit omega}(1).

(2) In \cite[Theorems 2.1 and 2.4]{Sh:64}, $\diamondsuit(\omega_1)$ is shown to be consistent with the existence of a 
stationary subset $S\s\omega_1$ on which $\unif(S,2)$ holds.
Now, appeal to Fact~\ref{fact216}(1) and Corollary~\ref{Cor - regressive Souslin implies almost disjoint clubsuit omega}(1).
\end{proof}

We conclude this subsection by stating an additional result, this time concerning the three-cardinal variant of $\clubsuit_{\ad}$ (recall Definition~\ref{threecardinalsvariant}):

\begin{theorem} Suppose that there exists a $\kappa$-Souslin tree.
Then $\clubsuit_{\ad}(\mathcal S,1,1)$ holds for some collection $\mathcal S$ of $\kappa$-many pairwise disjoint stationary subsets of $\kappa$.
If $\kappa$ is a successor cardinal, then moreover $\bigcup\mathcal S=E^\kappa_\chi$ for some cardinal $\chi\in\reg(\kappa)$.\qed
\end{theorem} 
We omit the proof due to the lack of applications of $\clubsuit_{\ad}(\mathcal S,1,1)$, at present.

\subsection{The consistency of the negation of our guessing principle}
By Lemma~\ref{lemma216}, for a stationary set $S$ consisting of points of some fixed cofinality,
$\clubsuit(S)$ entails $\clubsuit_{\ad}(\{S\},{<}\omega)$. The next theorem shows that the restriction to a fixed cofinality is crucial.
\begin{theorem} If $\kappa$ is weakly compact, then for any $S$ with $\reg(\kappa)\s S\s\kappa$, $\clubsuit_{\ad}(\{S\},1,1)$ fails.
\end{theorem}
\begin{proof} Suppose that $\vec A=\langle A_\alpha\mid\alpha\in S\rangle$ is a $\clubsuit_{\ad}(\{S\},1,1)$-sequence,
with $\reg(\kappa)\s S\s\kappa$.
We define a $C$-sequence $\langle C_\alpha\mid\alpha<\kappa\rangle$, as follows:

$\br$ Let $C_0:=\emptyset$.

$\br$ For every $\alpha<\kappa$, let $C_{\alpha+1}:=\{\alpha\}$.

$\br$ For every $\alpha\in\acc(\kappa)\setminus\reg(\kappa)$, 
fix a closed and cofinal $C_\alpha\s\alpha$ with $\otp(C_\alpha)=\cf(\alpha)<\min(C_\alpha)$.

$\br$ For every $\alpha\in\reg(\kappa)$, let $B_\alpha:=\nacc(A_\alpha)$ and finally let $C_\alpha:={B_\alpha}\uplus{\acc^+(B_\alpha)}$.
Note that $B_\alpha$ is a cofinal subset of $A_\alpha$, and that $C_\alpha$ is a closed and cofinal subset of $\alpha$.

Towards a contradiction, suppose that $\kappa$ is weakly compact.
So, by \cite[Theorem~1.8]{TodActa}, we may fix a club $C\s\kappa$ such that, for every $\delta<\kappa$,
for some $\alpha(\delta)<\kappa$, $C\cap\delta=C_{\alpha(\delta)}\cap\delta$.
Consider the club $D:=\{\delta\in\acc(\kappa)\mid \otp(C\cap\delta)=\delta\}$.
\begin{claim} For every $\delta\in D$, $\alpha(\delta)\in \reg(\kappa)$.
\end{claim}
\begin{proof} Let $\delta\in D$. Since $C_{\alpha(\delta)}\cap\delta=C\cap\delta$ is infinite, $\alpha(\delta)\in\acc(\kappa)$.
Now, if $\alpha(\delta)\in\acc(\kappa)\setminus\reg(\kappa)$, then $\delta=\otp(C_{\alpha(\delta)}\cap\delta)\le\otp(C_{\alpha(\delta)})<\min(C_{\alpha(\delta)})$,
which is an absurd.
\end{proof}

Evidently, $B:=\nacc(C)$ is a cofinal subset of $\kappa$.
So, by the choice of $\vec A$,
$G:=\{\delta\in S\cap D\mid \sup(B\cap A_\delta)=\delta\}$ is stationary.
Pick a pair of ordinals $\delta_0<\delta_1$ from $G$.
For each $i<2$, since $\alpha(\delta_i)\in\reg(\kappa)$,
$$B\cap\delta_i=\nacc(C)\cap\delta_i=\nacc(C_{\alpha(\delta_i)})\cap\delta_i= B_{\alpha(\delta_i)}\cap\delta_i\s A_{\alpha(\delta_i)}.$$

As $\sup(B\cap A_{\delta_i})=\delta_i$ and $B\cap\delta_i\s A_{\alpha(\delta_i)}$, $\sup(A_{\alpha(\delta_i)}\cap A_{\delta_i})=\delta_i$, 
so, since $\vec A$ is an AD-ladder system, it must be the case that $\alpha(\delta_i)=\delta_i$.
Altogether, $B\cap\delta_0\s A_{\delta_0}\cap A_{\delta_1}$, so that $\delta_0>\sup( A_{\delta_0}\cap A_{\delta_1})\ge\sup(B\cap\delta_0)=\delta_0$. This is a contradiction.
\end{proof} 
An ideal $\mathcal I$ consisting of countable sets is said to be a \emph{P-ideal}
iff every countable family of sets in the ideal admits a pseudo-union in the ideal. That is,
for every sequence $\langle X_n \mid n<\omega\rangle$ of elements of $\mathcal I$,
there exists $X\in\mathcal I$ such that $X_n\setminus X$ is finite for all $n<\omega$.

\begin{definition}[Todor\v{c}evi\'{c}, \cite{MR1809418}] The P-ideal dichotomy ($\pid$) asserts that for every P-ideal $\mathcal I$ consisting of countable subsets of some set $Z$, either:
\begin{enumerate}
\item there is an uncountable $B\s Z$ such that $[B]^{\aleph_0}\s\mathcal I$, or
\item there is a sequence $\langle B_n\mid n<\omega\rangle$ such that $\bigcup_{n<\omega}B_n=Z$ and, for each $n<\omega$, $[B_n]^{\aleph_0}\cap\mathcal I=\emptyset$.
\end{enumerate}
\end{definition}

We denote by $\pid_{\aleph_1}$ the restriction of the above principle to $Z:=\aleph_1$.
This special case was first introduced and studied by Abraham and Todor\v{c}evi\'{c} in \cite{MR1441232}, and was denoted there by $({}^\ast)$.

\begin{lemma}\label{pfa implies clubsuit_AD(omega_1) fails} Suppose that $\pid_{\aleph_1}$ holds and $\mathfrak b>\omega_1$.
Then, for any stationary $S\s \omega_1$, $\clubsuit_{\ad}(\{S\},1,1)$ fails.
\end{lemma}
\begin{proof} Towards a contradiction, suppose that $S\s\omega_1$ is stationary, and that $\vec A=\langle A_\alpha \mid\alpha\in S\rangle$ is a $\clubsuit_{\ad}(\{S\},1,1)$-sequence.
Let $$\mathcal I:=\{ X\in[\omega_1]^{\le\aleph_0}\mid \forall\alpha\in\acc(\omega_1)\cap S[A_\alpha\cap X\text{ is finite}]\}.$$
It is clear that $\mathcal I$ is an ideal.

\begin{claim} $\mathcal I$ is a P-ideal.
\end{claim}
\begin{proof} Let $\vec X=\langle X_n \mid n<\omega\rangle$ be a sequence of elements of $\mathcal I$. 
We need to find a pseudo-union of $\vec X$ that lies in $\mathcal I$.
As $\mathcal I$ is downward closed, we may assume that $\langle X_n\mid n<\omega\rangle$ consists of pairwise disjoint sets.

Fix a bijection $e:\omega\leftrightarrow\biguplus_{n<\omega}X_n$. Then, for all $\alpha\in S$, define a function $f_\alpha:\omega\rightarrow\omega$ via
$$f_\alpha(n):=\min\{ m<\omega\mid X_n\cap A_\alpha\s e``m\}.$$
As $\mathfrak b>\omega_1$, let us fix a function $f:\omega\rightarrow\omega$ such that $f_\alpha<^* f$ for all $\alpha\in S$.
Set $X:=\biguplus\{ X_n\setminus e[f(n)]\mid n<\omega\}$. Clearly, for every $n<\omega$, $X_n\setminus X$ is a subset of $e[f(n)]$, and, in particular, it is finite.

Towards a contradiction, suppose that $X\notin\mathcal I$. Fix $\alpha\in\acc(\omega_1)\cap S$ such that $A_\alpha\cap X$ is infinite.
Since $X\s\biguplus_{n<\omega}X_n$, but $A_\alpha\cap X_n$ is finite for all $n<\omega$, we may find a large enough $n<\omega$ 
such that $A_\alpha\cap X\cap X_n\neq\emptyset$ and $f_\alpha(n)<f(n)$.
Pick $\beta\in A_\alpha\cap X\cap X_n$. 
By the definition of $f_\alpha$, $\beta\in e[f_\alpha(n)]$. But $f_\alpha(n)<f(n)$,
so that $\beta\in e[f(n)]$, contradicting the fact that $\beta\in X$.
\end{proof}

\begin{claim} Let $B\s\omega_1$ be uncountable.
\begin{enumerate}
\item There exists $X\in[B]^{\aleph_0}$ with $X\notin\mathcal I$;
\item There exists $X\in[B]^{\aleph_0}$ with $X\in\mathcal I$.
\end{enumerate}
\end{claim}
\begin{proof} As $B$ is uncountable, $G:=\{\alpha\in \acc(\omega_1)\cap S\mid \sup(A_\alpha\cap B)=\alpha\}$ is stationary.

(1) Fix arbitrary $\alpha\in G$.
Then $X:=A_\alpha\cap B$ is an element of $[B]^{\aleph_0}\setminus\mathcal I$. 

(2) Let $\langle \alpha_n\mid n<\omega\rangle$ be some increasing sequence of elements of $G$. 
For every $n<\omega$, let $\langle \alpha^m_n\mid m<\omega\rangle$ be the increasing enumeration of some cofinal subset of $A_{\alpha_n}\cap B$.
Furthermore, we require that, for all $n<\omega$, $\alpha_n<\alpha_{n+1}^0$.

Set $\beta:=\sup_{n<\omega}\alpha_n$.
As $\vec A$ is an $\ad$-ladder system, for every $\alpha\in S\cap\acc(\omega_1)\setminus\beta$, 
we may define a function $f_\alpha:\omega\rightarrow\omega$ via:
$$f_\alpha(n):=\min\{ m<\omega\mid A_{\alpha_n}\cap A_\alpha\s \alpha^m_n\}.$$

As $\mathfrak b>\omega_1$, let us fix a function $f:\omega\rightarrow\omega$ such that $f_\alpha<^* f$ for all $\alpha\in S$.
Set $X:=\{ \alpha_n^{f(n)}\mid 0<n<\omega\}$. For every $n<\omega$, the interval $(\alpha_n,\alpha_{n+1})$ contains a single element of $X$,
so that $X$ is a cofinal subset of $\beta$ with $\otp(X)=\omega$.
In particular, $X\in[B]^{\aleph_0}$. 

Towards a contradiction, suppose that $X\notin\mathcal I$. Fix $\alpha\in\acc(\omega_1)\cap S$ such that $A_\alpha\cap X$ is infinite.
Clearly, $\alpha\ge\beta$. So, we may find $k<\omega$ such that, for every integer $n\ge k$, $f_\alpha(n)<f(n)$.
As $A_\alpha\cap X$ is infinite, let us now pick a positive integer $n\ge k$ such that $\alpha_n^{f(n)}\in A_\alpha$.
Recalling that $\{\alpha_n^m\mid m<\omega\}\s A_{\alpha_n}$,
we altogether infer that $\alpha_n^{f(n)}\in A_{\alpha_n}\cap A_\alpha\s \alpha_n^{f_\alpha(n)}$.
In particular, $\alpha_n^{f(n)}<\alpha_n^{f_\alpha(n)}$,
contradicting the fact that $f_\alpha(n)<f(n)$.
\end{proof}

Altogether, $\mathcal I$ is a P-ideal for which the two alternatives of $\pid_{\aleph_1}$ fail. This is a contradiction.
\end{proof}

\begin{cor} \begin{enumerate}
\item  $\pfa$ implies that $\clubsuit_{\ad}(\{\omega_1\},1,1)$ fails;
\item  $\clubsuit_{\ad}(\{\omega_1\},1,1)$ does not follow from the existence of an almost Souslin tree.
\end{enumerate}
\end{cor}
\begin{proof} (1) It is well-known that $\pfa$ implies $\pid+\ma_{\aleph_1}$. In particular, it implies $\pid_{\aleph_1}$ together with $\mathfrak b>\omega_1$.

(2) Almost Souslin trees were defined in \cite[\S3]{MR548979}.
In \cite{MR4128470} and \cite{MR4105604} one can find models of $\pid$ with $\mathfrak p>\omega_1$ (in particular,  $\mathfrak b>\omega_1$),
in which there exists an Aronszajn tree which is almost Souslin. 
\end{proof}
\begin{question}\label{mavsad}
	Does $\ma_{\aleph_1}$ imply that $\clubsuit_{\ad}(\omega_1)$ fails?
\end{question}
\begin{question} Is $\ch$ consistent with the failure of $\clubsuit_{\ad}(\omega_1)$?
\end{question}
Note that a combination of the main results of \cite{Juhasz_clubsuit_ostaszewski,CH_with_no_Ostaszweski_spaces}
implies that  $\ch$ is consistent with the failure of $\clubsuit_J(\omega_1)$.

\section{A Ladder-system Dowker space}\label{sectionladdersystemspace}

In \cite{Good_Dowker_large_cardinals}, Good constructed a Dowker space of size $\kappa^+$ using $\clubsuit(S,2),$ where $S$ is a non-reflecting stationary subset of $E^{\kappa^+}_{\omega}$.
We won't define the principle $\clubsuit(S,2)$,
but only mention that, by Fact~\ref{clubfacts}\eqref{matrix clubsuit from one}, it is no stronger than $\clubsuit(S)$. 
In this section, a ladder-system Dowker space of size $\kappa $ is constructed under the assumption of $\clubsuit_{\ad}(\mathcal S,1,2)$, where $\mathcal S$ is an infinite partition of some non-reflecting stationary subset of $\kappa$. 
By Lemma~\ref{lemma216}, $\clubsuit(S)$ implies $\clubsuit_{\ad}(\mathcal S,{<}\omega)$, which surely implies $\clubsuit_{\ad}(\mathcal S,1,2)$,
so, our construction in particular gives a ladder-system Dowker space in Good's scenario.
It also gives ladder-system Dowker spaces in scenarios considered by Rudin \cite{kappa_Dowker_from_souslin_Rudin} and Weiss \cite{MR628595}, 
as explained at the end of this section.

\medskip

The constructions in this and in the next section are motivated by the following lemma.

\begin{lemma}\label{Lemma - Dowker general argument}
	Suppose that $\mathbb X=( X,\tau)$ is a normal Hausdorff topological space of size $\kappa$,
	having no two disjoint closed sets of size $\kappa$.
	If there exists a $\s$-decreasing sequence $\langle D_n \mid n<\omega \rangle$ of closed sets of cardinality $\kappa$ such that $\bigcap_{n<\omega} D_n = \emptyset$,
	then $\mathbb X$ is Dowker.
\end{lemma}
\begin{proof}
Recall that, by \cite{Dowker_C.H.}, a space is Dowker iff it is Hausdorff, normal and 
there is a $\s$-decreasing sequence $\langle D_n \mid n<\omega \rangle$ of closed sets with $ \bigcap_{n<\omega}D_n=\emptyset $,
such that, for every sequence $\langle U_n \mid  n<\omega \rangle$ of open sets, if $ D_n\subseteq U_n $ for every $n<\omega$, 
then $ \bigcap_{n<\omega}U_n\neq\emptyset $.

Now suppose that there exists a $\s$-decreasing sequence $\langle D_n \mid n<\omega \rangle$ of closed sets of cardinality $\kappa$ such that $\bigcap_{n<\omega} D_n = \emptyset$.
Suppose that $\langle U_n \mid n<\omega\rangle $ is a sequence of open sets such that $ D_n\subseteq U_n $ for every $n<\omega$.
	
For every $n<\omega$, $F_n:=X\setminus U_n$ is a closed set disjoint from $D_n$, and hence of cardinality $<\kappa$.
So, as $\omega<\cf(\kappa)=\kappa$, $\bigcup_{n<\omega}F_n$ has size less then $\kappa$.
In particular, $X\setminus \bigcup_{n<\omega}F_n\neq\emptyset$.
Altogether, $ \bigcap_{n<\omega}U_n=\bigcap_{n<\omega} X\setminus F_n=X\setminus \bigcup_{n<\omega} F_n\neq\emptyset$.
Recalling that $\mathbb X$ is normal, we altogether infer that $\mathbb X$ is Dowker.
\end{proof}

\medskip\noindent\textbf{The space.} Suppose that $S$ is a stationary subset of $\acc(\kappa)$,
$\mathcal S$ is a partition of $S$ with $|\mathcal S|=\aleph_0$,
and $\clubsuit_{\ad}(\mathcal S,1,2)$ holds.
Fix an $\ad$-ladder system $\langle A_{\alpha}\mid\alpha\in S\rangle$ as in Lemma~\ref{adone}.
Fix an injective enumeration $\langle S_{n+1} \mid n<\omega \rangle $ of $\mathcal S$, and let $S_0:=\kappa\setminus S$. 
As $\langle S_n\mid n<\omega\rangle$ is a partition of $\kappa$,
for each $\alpha<\kappa$, we may let $n(\alpha)$ denote the unique $n<\omega$ such that $\alpha\in S_n$.
For each $n<\omega$, let $ W_n:=\bigcup_{i\leq n}S_i $. 
Finally, define a sequence $\vec L=\langle L_\alpha\mid\alpha<\kappa\rangle$ via:
$$L_\alpha:=\begin{cases}
W_{n(\alpha)-1}\cap A_\alpha,&\text{if }n(\alpha)>0\ \&\ \sup(W_{n(\alpha)-1}\cap A_\alpha)=\alpha;\\
\emptyset,&\text{otherwise.}
\end{cases}$$

\begin{lemma}\label{lemma61}
	\begin{enumerate}
		\item\label{lemma - ladder-system club sequence - Clause W_n is open} For all $n<\omega$ and $\alpha\in S_{n+1}$, $L_\alpha\s W_n$;
			\item $\bar S:=\{ \alpha\in\acc(\kappa)\mid \sup(L_\alpha)=\alpha\}$ is a stationary subset of $S$;
		\item\label{adofladders} For all $\alpha\neq\alpha'$ from $\bar S$, $\sup(L_\alpha\cap L_{\alpha'})<\alpha$;
		\item\label{lemma - ladder-system club sequence - Clause hitting} For all $B_0,B_1\in[\kappa]^\kappa$, there exists  $m<\omega$ such that, for every $n\in\omega\setminus m$, the following set is stationary:
			\[ \{ \alpha\in S_n \mid \sup(L_\alpha\cap B_0)=\sup(L_\alpha\cap B_1)=\alpha]  \}.\]
	\end{enumerate}
\end{lemma}
\begin{proof} (2) This follows from Clause~(4).

(3) For all $\alpha\neq\alpha'$ from $\bar S$, $\sup(L_\alpha\cap L_{\alpha'})\le\sup(A_\alpha\cap A_{\alpha'})<\alpha$;

(4) Given two cofinal subsets $B_0,B_1$ of $\kappa$, 
find $m_0,m_1<\omega$ be such that $|B_0\cap S_{m_0}|=|B_1\cap S_{m_1}|=\kappa$. Evidently, $m:=\max\{n_0,n_1\}+1$ is as sought.
\end{proof}

Now, consider the ladder-system space $\mathbb X=(\kappa,\tau)$ which is determined by $\vec L$ (equivalently, determined by $\vec L\restriction\bar S$).
This means that a set $U\s\kappa$ is $\tau$-open iff, for every $\alpha\in U$, $L_\alpha\s^* U$.
Put differently, if we denote $N^\epsilon_\alpha:=(L_\alpha\setminus\epsilon)\cup\{\alpha\}$,
then, for every $\alpha<\kappa$, every neighborhood of $\alpha$ covers some element from
$\mathcal N_\alpha:=\{ N^{\epsilon}_\alpha\mid \epsilon<\alpha \}$.
\begin{definition}For any set of ordinals $N$, denote $N^-:=N\cap\sup(N)$.
\end{definition}
Note that, for all $\alpha<\kappa$ and $N\in\mathcal N_\alpha$, $N^-$ is either empty or a cofinal subset of $\alpha$.

\begin{lemma}\label{t1lemma}
	The space $\mathbb X$ is $ T_1 $.
\end{lemma}
\begin{proof}
	Let $x$ be an element of the space $\mathbb X$.
	To see that $\kappa\setminus \{x\}$ is $\tau$-open,
	notice that for every $y\in X$, $\mathcal N_y$ is a chain such that $\bigcap\mathcal N_y=\{y\}$.
	In particular, for every $y\in X\setminus\{x\}$,
	there exists $N_y\in\mathcal N_y$ with $N_y\s\kappa\setminus\{x\}$.
\end{proof}

\begin{lemma}\label{lemma35}
\begin{enumerate}
\item $\kappa\setminus\bar S$ is a discrete subspace of size $\kappa$;
\item For every $\xi<\kappa$, $\xi$ is $\tau$-open;
\item For every $\xi\in\kappa\setminus\bar S$, $\xi$ is $\tau$-closed;
\item For every $n<\omega$, $W_n$ is $\tau$-open.
\end{enumerate}
\end{lemma}
\begin{proof} (4) By Lemma~\ref{lemma61}\eqref{lemma - ladder-system club sequence - Clause W_n is open}.
\end{proof}

\begin{lemma}\label{another lambda^+ dowker - no two disjoint big closed sets clubsuit dowker}
	There are no two disjoint $\tau$-closed subsets of cardinality $\kappa$.
\end{lemma}
\begin{proof} Towards a contradiction,
	suppose that $K_0$ and $K_1$ are two disjoint $\tau$-closed subsets of cardinality $\kappa$. 
	Using Lemma~\ref{lemma61}\eqref{lemma - ladder-system club sequence - Clause hitting},
	let us fix $n<\omega$ such that  $\sup(L_\alpha\cap K_0)=\sup(L_\alpha\cap K_1)=\alpha$.
	As both $K_0$ and $K_1$ are $\tau$-closed, this implies that $\alpha\in K_0$ and $\alpha\in K_1$, contradicting the fact $K_0$ and $K_1$ are disjoint.
\end{proof}

Following the terminology coined in \cite{MR44519} and \cite{MR2099600}, we introduce the following.

\begin{definition} The sequence $\vec L$ is said to be \emph{almost $\mathcal P_0$} iff, for every $\xi<\kappa$ and every function $c:\bar S\cap\xi\rightarrow\omega$,
there exists a function $c^*:\xi\rightarrow\omega$ such that, for every $\alpha\in\bar S\cap\xi$, $c^*\restriction L_\alpha$ is eventually constant with value $c(\alpha)$.
\end{definition}

\begin{lemma}\label{normal clubsuit dowker} If $\vec L$ is almost $\mathcal P_0$, then $\mathbb X$ is normal and Hausdorff.
\end{lemma}
\begin{proof} Suppose that $\vec L$ is almost $\mathcal P_0$.
	Let $K_0$ and $K_1$ be disjoint $\tau$-closed subsets of $\kappa$.
	By Lemma~\ref{another lambda^+ dowker - no two disjoint big closed sets clubsuit dowker}, at least one of them is bounded, say $K_0$.
	Using Lemma~\ref{lemma35}(1), fix a large enough $\xi\in\kappa\setminus\bar S$ such that $K_0\subseteq \xi$.
	Note that by Clauses (2) and (3) of Lemma~\ref{lemma35}, $\xi$ is clopen.
	Now, set $K^0_0:=K_0$ and $K^0_1:=K_1$.
	
	\begin{claim} Suppose $n<\omega$ and that $K^n_0$ and $K^n_1$ are disjoint $\tau$-closed sets with $K^n_0\s\xi$. 
	Then there exist disjoint $\tau$-closed sets $K^{n+1}_0$ and $K^{n+1}_1$ with $K^{n+1}_0\s\xi$ such that for all $i<2$:
	\begin{enumerate}
	\item $K^n_i\s K^{n+1}_i$;
	\item for every $\alpha\in K^n_i\cap\xi$, $L_\alpha\s^* K^{n+1}_i$.
	\end{enumerate}
	\end{claim}
	\begin{proof} 
	For every $i<2$, define $c_i:\bar S\cap\xi\rightarrow2$ via $c_i(\alpha):=1$ iff $\alpha\in K^n_i$.
	Now, as $\vec L$ is almost $\mathcal P_0$, for each $i<2$, we may fix a function $c^*_i:\xi\rightarrow2$ 
	such that, for every $\alpha\in\bar S\cap \xi$, $c_i^*\restriction L_\alpha$ is eventually constant with value $c_i(\alpha)$.
	For each $i<2$, set $$K^{n+1}_i:=K^n_i\cup\{\gamma\in\xi\setminus K^n_{1-i}\mid c^*_i(\gamma)=1\ \&\ c^*_{1-i}(\gamma)=0\}.$$
	Evidently $K_i^n\s K_i^{n+1}$. It is also easy to see that $K^{n+1}_0\cap K^{n+1}_1=\emptyset$.

	Let $i<2$. To see that $K_i^{n+1}$ is $\tau$-closed,
	fix an arbitrary nonzero $\alpha<\kappa$ such that $\sup(K_i^{n+1}\cap L_\alpha)=\alpha$,
	and we shall prove that $\alpha\in K_i^{n+1}$. As $K_i^n\s K_i^{n+1}$, we may avoid trivialities, and assume that $\alpha\notin K_i^{n}$,
	so that $\alpha$ belongs to the set
	$$\{\gamma\in\xi\setminus K^n_{1-i}\mid c^*_i(\gamma)=1\ \&\ c^*_{1-i}(\gamma)=0\}.$$
	Altogether, $\alpha\in(\bar S\cap \xi)\setminus K_i^n$, which must mean that $c_i(\alpha)=0$.
	Fix a large enough $\epsilon<\alpha$ such that $c_i^*\restriction(L_\alpha\setminus\epsilon)$ is eventually constant with value $0$.
	Then $\sup(K_i^{n+1}\cap L_\alpha)\le\epsilon<\alpha$, contradicting the choice of $\alpha$.
	
	Finally, to verify Clause~(2), fix arbitrary $i<2$ and $\alpha\in K_i^n\cap\xi\cap\bar S$.
	By the definition of the two functions, $c_i(\alpha)=1$ and $c_{1-i}(\alpha)=0$.
	So, there exists a large enough $\epsilon<\alpha$ such that $c_i^*\restriction(L_\alpha\setminus\epsilon)$ 
	is a constant function with value $1$,
	and $c_{1-i}^*\restriction(L_\alpha\setminus\epsilon)$ 
	is a constant function with value $0$. In effect, $L_\alpha\setminus\epsilon\s K_i^{n+1}$.	
	\end{proof}

	By an iterative application of the preceding claim, we obtain a sequence of pairs $\langle (K_0^n,K_1^n)\mid n<\omega\rangle$.
	Set $U_0:=\bigcup_{n<\omega}K_0^n$ and $U_1:=(\kappa\setminus \xi)\cup\bigcup_{n<\omega}K_1^n$.
	By Clause~(2) of the preceding claim and the fact that $\xi$ is clopen, $U_i$ is open for each $i<2$.
	Thus, we are left with verifying the following.

\begin{claim} $U_0\cap U_1=\emptyset$.
\end{claim}
\begin{proof} Suppose not, and pick $\alpha\in U_0\cap U_1$.
Notice that $U_0\subseteq \xi$, hence we can find $n_0,n_1<\omega$ such that $\alpha\in K_0^{n_0}\cap K^{n_1}_1$. 
Then, for $n:=\max\{n_0,n_1\}$, we get that $\alpha\in K_0^n\cap K_1^n$, contradicting the fact that $K_0^n$ and $K_1^n$ are disjoint.
\end{proof}

This completes the proof of normality. Since $\mathbb X$ is $T_1$, it also follows that it is Hausdorff.
\end{proof}

\begin{cor}\label{clubsuit_AD_space_not_countably_paracompact}
	If $\vec L$ is almost $\mathcal P_0$, then $\mathbb X$ is Dowker. 
\end{cor}
\begin{proof} 	
	For every $n<\omega$, set $D_n:=\kappa\setminus W_n$. 
	Then $\langle D_n \mid n<\omega \rangle$ is a  $\s$-decreasing sequence of closed sets of cardinality $\kappa$ such that $\bigcap_{n<\omega} D_n = \emptyset$.
	So, by Lemmas \ref{normal clubsuit dowker}, \ref{another lambda^+ dowker - no two disjoint big closed sets clubsuit dowker} 
	and \ref{Lemma - Dowker general argument}, $\mathbb X$ is Dowker.
\end{proof}

\begin{lemma}\label{lemma62} Each of the following two imply that $\vec L$ is almost $\mathcal P_0$:
\begin{enumerate}
\item $\ma({<}\kappa)$ holds and $\otp(L_\alpha)=\omega$ for all $\alpha\in\bar S$;
\item $\bar S$ is a non-reflecting stationary set.
\end{enumerate}
\end{lemma}
\begin{proof}
(1) This follows immediately from \cite[{\S}II, Theorem~4.3]{Proper_and_improper_forcing_Shelah_Book}.

(2) By Lemma~\ref{lemma61}\eqref{adofladders} and Proposition~\ref{Lemma - non-reflecting stat, diagonlization of initial seg},
we may fix a sequence $\langle f_\xi\mid \xi<\kappa\rangle$ such that, for every $\xi<\kappa$:
	\begin{itemize}
	\item $f_\xi$ is a regressive function from $\bar S\cap\xi$ to $\xi$;
	\item the sets in $\langle L_\alpha\setminus f_\xi(\alpha) \mid \alpha\in\bar S\cap \xi\rangle$ are pairwise disjoint.
	\end{itemize}
	
	Now, given a nonzero $\xi<\kappa$, let $f^+_\xi:\xi\rightarrow\xi$ denote the function
	such that $f^+_\xi(\alpha)=f_\xi(\alpha)$ for all $\alpha\in\bar S\cap\xi$,
	and $f^+_\xi(\alpha)=0$ for all $\alpha\in\xi\setminus\bar S$. 
	Evidently, for every $\beta<\xi$, $\{ \alpha<\xi\mid \beta\in L_\alpha\setminus f^+_\xi(\alpha)\}$ is a subset of $\bar S$ containing at most one element.
	So, for any function $c:\bar S\cap\xi\rightarrow\omega$, we may define a corresponding function $c^*:\xi\rightarrow\omega$ via:
	$$c^*(\beta):=\begin{cases}
	c(\alpha),&\text{if }\beta\in L_\alpha\setminus f^+_\xi(\alpha);\\
	0,&\text{otherwise}.
	\end{cases}$$
	A moment's reflection makes it clear that $c^*$ is as sought.
\end{proof}

\begin{cor} Suppose that $\mathcal S$ is an infinite partition of some non-reflecting stationary subset of a regular uncountable cardinal $\kappa$.
If $\clubsuit_{\ad}(\mathcal S,1,2)$ holds,
then there exists a ladder-system Dowker space of cardinality $\kappa$.\qed
\end{cor}
\begin{remark} The preceding is the Introduction's Theorem~B.
\end{remark}

\begin{cor} If ${\ma}+  \clubsuit(E^{\mathfrak c}_\omega)$ holds, then there exists a ladder-system Dowker space over $\mathfrak c$.

In particular, if $\ma$ holds and $\mathfrak c$ is the successor of a cardinal of uncountable cofinality,
then there exists a ladder-system Dowker space over $\mathfrak c$.
\end{cor}
\begin{proof} The first part follows from Lemmas \ref{lemma216} and \ref{lemma62}(1).
For the second part, note that if $\ma$ holds and $\mathfrak c=\lambda^+$,
then $2^\lambda=\lambda^+$, so  if, moreover, $\lambda$ is a cardinal of uncountable cofinality,
then by the main result of \cite{Sh:922}, $\diamondsuit(E^{\lambda^+}_\omega)$ holds. In particular, in this case, $\clubsuit(E^{\mathfrak c}_\omega)$ holds.
\end{proof}
\begin{remark} In \cite{MR628595}, Weiss proved that if $\ma$ and $\diamondsuit(E^{\mathfrak c}_\omega)$ both hold, then there exists a locally compact, first countable, separable, real compact, Dowker space of size $\mathfrak c$.
\end{remark}

\begin{cor}
	If there exists a $\lambda$-complete $\lambda^+$-Souslin tree,
	then there exists a ladder-system Dowker space over $\lambda^+$.
\end{cor}
\begin{proof}
	By Corollary~\ref{Souslin implies almost disjoint clubsuit}, Lemma~\ref{lemma62}(2), and
	the fact that $E^{\lambda^+}_\lambda$ is a non-reflecting stationary subset of $\lambda^+$.
\end{proof}
\begin{remark} In \cite{kappa_Dowker_from_souslin_Rudin}, Rudin constructed a Dowker space of size $\lambda^+$ from a $\lambda^+$-Souslin tree, for $\lambda$ regular.
\end{remark}

\begin{cor}
If there exist a regressive $\kappa$-Souslin tree and a non-reflecting stationary subset of $E^\kappa_\omega$,
then there exists a ladder-system Dowker space over $\kappa$.
\end{cor}
\begin{proof}
	By Corollary~\ref{Cor - regressive Souslin implies almost disjoint clubsuit omega} and Lemma~\ref{lemma62}(2).
\end{proof}
\begin{remark} It is well-known (see \cite{Jensen_V=L_Diamond} or \cite{paper22}) that if $\kappa>\aleph_0$, $\square_\kappa$ holds and $2^\kappa=\kappa^+$,
then there exists a regressive $\kappa$-Souslin tree and there exists a non-reflecting stationary subset of $E^{\kappa^+}_\omega$.
In \cite{Good_Dowker_large_cardinals}, 
Good proved that if $\kappa>\aleph_0$, $\square_\kappa$ holds and $2^\kappa=\kappa^+$,
then there exists a Dowker space over $\kappa^+$ which is first countable, locally countable, locally compact, zero-dimensional, and collectionwise normal that is of scattered length $\omega$.
\end{remark}

\section{de Caux type spaces}\label{Collectionwisesection}
\subsection{Collectionwise normality}
In this short subsection, we present a sufficient condition for a certain type of topological space to be collectionwise normal.
This will be used in verifying that the spaces constructed later in this section are indeed collectionwise normal.

\begin{definition} Let $\mathbb X=(X,\tau)$ be a topological space.
\begin{itemize}
\item A sequence $\langle K_i\mid i<\theta\rangle$ of subsets of $ X $ is said to be \emph{discrete} iff for every $ x\in X $, there is an open neighborhood $U$ of $x$ such that $\{ i<\theta\mid U\cap K_i\neq\emptyset\}$ contains at most one element.
	
\item The space $ \mathbb X $ is said to be \emph{collectionwise normal} iff for every discrete sequence $ \langle K_i \mid i<\theta\rangle $ of closed sets, 
there exists a sequence $ \langle U_i\mid i<\theta\rangle$ of pairwise disjoint open sets such that $ K_i\s U_i $ for all $i<\theta$.
\end{itemize}
\end{definition}
\begin{remark} Note that any collectionwise normal space is normal.
\end{remark}

Let $\mathbb X=(X,\tau)$ be some topological space determined by a sequence of weak neighborhoods, $\langle \mathcal N_x \mid x\in X \rangle$.
This means that a subset $ U \s X$ is $\tau$-open iff for any $x\in U $, there is $N\in\mathcal N_x$ with $N\s U$.

\begin{lemma}\label{general normal lemma}
	Suppose that $\theta$ is some nonzero cardinal and that $\langle K_{i} \mid {i}<\theta\rangle$ is a discrete sequence of $\tau$-closed sets,
	 $O$ is a $\tau$-open set covering $\bigcup_{0<i<\theta}K_i$,
	 and there exists a transversal $\langle N_x\mid x\in X\rangle\in\prod_{x\in X}\mathcal N_x$ such that:
	\begin{enumerate}[label=(\alph*)]
		\item for all $x\in O$, $N_{x}\subseteq O$;
		\item\label{disjoint tails general} for all $x\in O$ and $x'\in X\setminus\{x\}$, $N_{x}\cap N_{x'}\s\{x,x'\}$;
		\item\label{tails outside H,K general} for all $x\in X$ and ${i}<\theta$, if $N_x\cap K_{i}\neq \emptyset$, then $x\in K_{i}$.
	\end{enumerate}
	Then there exists a sequence $\langle U_{i} \mid {i}<\theta \rangle$ of pairwise disjoint $\tau$-open sets such that $K_0\subseteq U_0$,
	and $K_i\subseteq U_i\subseteq O$ for all nonzero $i<\theta$.
\end{lemma}
\begin{proof}
	By recursion on $n<\omega$, we construct a matrix $\langle U_i^n \mid {i}<\theta,~n<\omega \rangle $, as follows:
	\begin{itemize}
		\item[$\br$] For each ${i}<\theta$, set $U^0_{i}:=K_{i}$.
		\item[$\br$] For every $n<\omega$ such that $\langle U_i^n \mid {i}<\theta\rangle$ has already been defined, set
			$U^{n+1}_{{i}}:= \bigcup\{N_x \mid x  \in U_i^n\}$ for each $i<\theta$.
	\end{itemize}		
	
	Evidently, $U_{i}:=\bigcup_{n<\omega}U_i^n$ is an open set covering $K_{i}$.
	
	\begin{claim}\label{U is bounded by xi general} Let $i$ with $0<i<\theta$. Then $U_{i}\subseteq O$.
	\end{claim}
	\begin{proof} We have $U_{i,0}=K_{i}\s O$.
		In addition, for every $n<\omega$ such that $U_i^n\s O$,
		Clause~(a) implies that $N_x\subseteq O$ for every $x \in U_i^n$,
		and hence $U_{i}^{n+1}\s O$.
	\end{proof}
	\begin{claim}
		The sets in the sequence $\langle U_{i} \mid {i}<\theta \rangle$ are pairwise disjoint.
	\end{claim}	 
	\begin{proof} Suppose not. Fix ${i}\neq{i}'$ in $\theta$ such that $U_{{i}}\cap U_{i'}\neq \emptyset $.
		Let $n:=\min \{ k<\omega \mid U_{{i}}^k\cap U_{{i}'} \neq \emptyset \}$,
		and then let $n':=\min\{ k<\omega \mid U_i^n\cap U_{{i}'}^k\neq \emptyset \}$.
		
		\begin{subclaim} $\min\{n,n'\}>0$.
		\end{subclaim}
		\begin{proof}
			First, as $U_{{i}}^0=K_{i}$ is disjoint from $U_{{i}'}^0=K_{{i}'}$, we infer that $(n,n')\neq(0,0)$.
			
			$\br$ If $n=0$ and $n'>0$, then let $y\in K_{i}\cap U_{{i}'}^{n'}$.
			It follows that there exists some $x \in U_{{i}'}^{n'-1}$ such that $y\in K_{i}\cap N_x $. 
			By Clause~\ref{tails outside H,K general}, then, $x \in K_{i}$.
			So $x \in U_{{i}}^0\cap U_{{i}'}^{n'-1}$, contradicting the minimality of $n'$.
			
			$\br$ If $n>0$ and $n'=0$, then let $y\in U_i^n\cap K_{{i}'}$.
			It follows that there exists some $x \in U_{{i}}^{n-1}$ such that $N_x\cap K_{{i}'}\neq\emptyset$. 
			By Clause~\ref{tails outside H,K general}, then, $x \in K_{{i}'}$.
			So $x \in U_{{i}}^{n-1}\cap U_{{i}'}^0$, contradicting the minimality of $n$.
		\end{proof}

		It follows that, for all $y\in U_i^n$, either $y\in U_{{i}}^{n-1}$ or $y\in N_x $ for some $x \in U_{{i}}^{n-1}\setminus\{y\}$.
		Likewise, for all $y\in U_{i'}^{n'}$, either $y\in U_{i'}^{n'-1}$ or $y\in N_x $ for some $x \in U_{i'}^{n'-1}\setminus\{y\}$.
		Now, by the choice of the pair $(n,n')$, let us fix $y\in U_i^n\cap U_{{i}'}^{n'}$. There are four cases to consider:
			
			\begin{enumerate}
				\item $y\in U_{{i}}^{n-1}\cap U_{{i}'}^{n'-1}$.
				In this case, $U_{{i}}^{n-1}\cap U_{{i}'} \neq \emptyset$, contradicting the minimality of $n$.
				\item $y\in N_{x}\cap U_{{i}'}^{n'-1}$ for some $x \in U_{{i}}^{n-1}\setminus\{y\}$.			
				In this case, $N_{x}\subseteq U_i^n$, so $U_i^n\cap U_{{i}'}^{n'-1} \neq \emptyset$, contradicting the minimality of $n'$.
				\item $y\in U_{{i}}^{n-1}\cap N_{x}$ for some $x\in U_{{i}'}^{n'-1}\setminus\{y\}$.
				In this case, $N_{x}\subseteq U_{{i}'}^{n'}$, so $U_{i}^{n-1}\cap U_{{i}'} \neq \emptyset$, contradicting the minimality of $n$.

				\item There exist $x \in U_{{i}}^{n-1}\setminus\{y\}$ and $x'\in U_{{i}'}^{n'-1}\setminus\{y\}$ such that $y \in N_{x}\cap N_{x'}$.				
				Equivalently, there exist $x \in U_{{i}}^{n-1}$ and $x'\in U_{{i}'}^{n'-1}$ such that $y \in(N_{x}\cap N_{x'})\setminus\{x,x'\}$.				
				In this case, there are two subcases:
				\begin{enumerate}
					\item[$\br$] If $x =x '$, then $x \in U_{i}^{n-1}\cap U_{{i}'}^{n'-1}$, contradicting the minimality of $n$.
					
					\item[$\br$] If $x \neq x '$, then, by Claim~\ref{U is bounded by xi general}, either $x  \in O$ or $x'\in O$. 
This is in contradiction with Clause~\ref{disjoint tails general}.
				\end{enumerate}
			\end{enumerate}

		Altogether, $\{U_{i}\mid {i}<\theta\}$ is a family of pairwise disjoint sets as sought.
	\end{proof}	
	This completes the proof.
\end{proof}

\subsection{$O$-space}\label{higherSspaceSection}
This subsection is dedicated to proving Theorem~C:

\begin{theorem}\label{ospace}
Suppose that $\clubsuit_{\ad}(\{\omega_1\},\omega,1)$ holds. Then there exists a collectionwise normal, 
non-Lindel\"of $O$-space.
\end{theorem}
\begin{remark}
Note that unlike Dahroug's construction that defines a topology over the $\aleph_1$-Souslin tree,
here the topology will be defined over $\omega_1$. Thus, when taken together with Corollary~\ref{w1souslin}, this appears to yield the first ``linear'' construction of an $O$-space from an $\aleph_1$-Souslin tree.
\end{remark}
\begin{remark} The arguments of this subsection immediately generalize to yield a collectionwise normal higher $O$-space from $\clubsuit_{\ad}(\{E^{\lambda^+}_\lambda\},\lambda,1)$.
The focus on the case $\lambda=\omega$ is just for simplicity.
\end{remark}
Let $\langle \mathcal A_{\alpha}\mid\alpha<\omega_1\rangle$ be a $\clubsuit_{\ad}(\{\omega_1\},\omega,1)$-sequence.
For each $\alpha\in\acc(\omega_1)$, fix an injective enumeration $\{ A_{\alpha+i}\mid i<\omega\}$ of the elements of $\mathcal A_\alpha$.
For every infinite $\xi<\omega_1$, as $\mathcal B_\xi:=\{ A_\beta\mid \omega\le\beta<\xi+\omega\}$ is a countable subset of $\bigcup_{\alpha\in E^{\omega_1}_\omega}\mathcal A_\alpha$,
we may appeal to Proposition~\ref{Proposition - disjointify multi-ladder system} to fix a function $f_\xi:\mathcal B_\xi\rightarrow\omega_1$ such that:
\begin{enumerate}
	\item for every $B\in\mathcal B_\xi$, $f_\xi(B)\in B$;
	\item the sets in $\langle B\setminus f_\xi(B) \mid B\in\mathcal B_\xi\rangle$ are pairwise disjoint.
\end{enumerate}		

\begin{definition} For every $\beta<\omega_1$, let $\alpha_\beta:=\min\{\alpha\le\beta\mid \exists i<\omega(\beta=\alpha+i)\}$.
\end{definition}

We are now ready to define our topological space $\mathbb X=(\omega_1,\tau) $.
For all $\beta<\omega$, let $\mathcal N_\beta:=\{\{\beta\}\}$.
For all infinite $\beta<\omega_1$ and $\epsilon<\alpha_\beta$, denote $N_\beta^\epsilon:=(A_{\beta}\setminus\epsilon)\cup\{\beta\}$,
and then set  $\mathcal N_\beta:=\{ N_\beta^\epsilon\mid \epsilon<\alpha_\beta\}$.
Now, a subset $ U \s \omega_1$ is $\tau$-open iff for any $\beta\in U $, 
there is $N\in\mathcal N_\beta$ with $N\s U$.	
It is easy to check that $\mathbb X $ is a $ T_1 $ topological space
and that, for every $\xi<\omega_1$, $\xi$ is $\tau$-open.

\begin{definition}
	For any set of ordinals $N$, denote $N^-:=N\cap\sup(N)$.
\end{definition}
Note that for all $\beta<\omega_1$ and $N\in\mathcal N_\beta$, $N^-$ is a cofinal subset of $\alpha_\beta$.

\begin{remark} The topology of the space from the previous section 
is such that a set $U$ is open iff, for every $\beta\in U$, $L_\beta\s^* U$,
and the topology of the space here is such that a set $U$ is open iff, for every $\beta\in U$, $A_\beta\s^* U$.
The approach seems identical, but there is a subtle difference:
in the previous section, for every ordinal $\beta$, we had $\sup(L_\beta)\in\{0,\beta\}$,
whereas here, for every ordinal $\beta$, we have $\sup(A_\beta)=\alpha_\beta\in\beta+1$.
\end{remark}

\begin{lemma}
For every $\alpha\in\acc(\omega_1)$, $\cl([\alpha,\alpha+\omega))=\omega_1\setminus\alpha$.
\end{lemma}
\begin{proof}
	
	Let $A:=\cl([\alpha,\alpha+\omega))$.
	We prove by induction on $ \delta\in\acc(\omega_1)\setminus\alpha $ that $ [\delta,\delta+\omega)\subseteq A$.
	
	$\br$ For $\delta=\alpha$, trivially $ [\alpha,\alpha+\omega) \subseteq A$.
	
	$\br$ Suppose $\delta\in\acc(\omega_1)\setminus\alpha$ and $[\delta,\delta+\omega)\s A$. 
	Put $\delta':=\delta+\omega$ and we shall prove that $[\delta',\delta'+\omega)\s A$.
	Let $\beta \in  [\delta',\delta'+\omega)$ be arbitrary.
	For each $N\in \mathcal N_\beta$, we have $\sup(N\cap [\delta,\delta+\omega))=\delta+\omega$, and hence $\beta\in\cl([\delta,\delta+\omega))\s A$.
	
	$\br$ Suppose that $\delta\in\acc(\acc(\omega_1)\setminus \alpha)$ and, for every $\gamma\in\acc(\delta)\setminus\alpha$, $[\gamma,\gamma+\omega)\s A$.
	Therefore, $[\alpha,\delta) \subseteq A$. 
	For every $\beta \in  [\delta,\delta+\omega)$ and $N\in \mathcal N_\beta$, $\sup(N\cap [\alpha,\delta))=\delta$,
	and hence $\beta\in\cl([\alpha,\delta))\s A$.
\end{proof}

\begin{cor}\label{closed sets are countable or co countable}
	\begin{enumerate}
		\item $(\omega_1,\tau)$ is hereditary separable;
		\item\label{Clause - S-space closed sets countable or co count.} Every $ \tau $-closed set is either countable or co-countable.
	\end{enumerate}
\end{cor}
\begin{proof}
	Let $B$ be an arbitrary uncountable subset of $\omega_1$.
	By Clause~(\ref{clubsuit AD, Clause cofinal set}) of Definition~\ref{clubsuit AD definition}, 
	we can find an $ \alpha\in\acc(\omega_1)$ such that $  \sup(A_{\alpha+i}\cap B)=\alpha $
	for all $ i<\omega$.
	Let $D:=B\cap\alpha$.
	If follows that, for every $\beta\in[\alpha,\alpha+\omega)$ and every $N\in\mathcal N_\beta$, 
	$\sup(N\cap D)=\alpha$. So, $ [\alpha,\alpha+\omega) \subseteq \cl(D)$
	and hence $\omega_1\setminus\alpha=\cl([\alpha,\alpha+\omega))\subseteq \cl(D)$.
	
	\begin{enumerate}
		\item As $B\setminus(\omega_1\setminus\alpha)=D$, it follows that $\cl(B)=\cl(D)$,
		so that $D$ is a countable dense subset of the subspace $B$.
		
		\item If $B$ is moreover closed, then $(\omega_1\setminus\alpha)\s B$, so that $B$ is co-countable.	\qedhere
	\end{enumerate}
\end{proof}

\begin{lemma} $\mathbb X$ is Hausdorff and collectionwise normal.
\end{lemma}
\begin{proof} As $\mathbb X$ is $T_1$, it suffices to verify that it is collectionwise normal.
	Let $\vec K=\langle K_i \mid i<\theta \rangle$ be an arbitrary discrete sequence of closed sets,
	for some cardinal $\theta$. To avoid trivialities, assume that $\theta\ge 2$.

	\begin{claim} $\theta\le\omega$.
	\end{claim}
	\begin{proof} Otherwise, set $B:=\{\beta_i\mid i<\omega_1\}$ for some transversal $\langle \beta_i\ \mid i<\omega_1\rangle\in\prod_{i<\omega_1}K_i$.
	As $B$ is necessarily uncountable, using Clause~(\ref{clubsuit AD, Clause cofinal set}) of Definition~\ref{clubsuit AD definition}, 
	we may fix $\alpha\in\acc(\omega_1)$ such that $\sup(A_\alpha\cap B)=\alpha$.
	Then any open neighborhood of $\alpha$ meets infinitely many elements of $\vec K$,
	contradicting its discreteness.
	\end{proof}

	By Corollary~\ref{closed sets are countable or co countable}\eqref{Clause - S-space closed sets countable or co count.}, 
	the sequence $\vec K$ contains at most one uncountable set.
	By possibly re-indexing, we may assume that $\{ i<\theta\mid |K_i|=\aleph_1\}\s\{0\}$.
	Now, as $\theta$ is countable,
	we may find a large enough $\xi\in\acc(\omega_1)$ such that $K_i\subseteq \xi$ for all nonzero $i<\theta$.
	
	\begin{claim}\label{ostaszewski tails claim}
		There exists a sequence $\langle N_\beta\mid \beta<\omega_1\rangle\in\prod_{\beta<\omega_1}\mathcal N_\beta$
		such that:
		\begin{enumerate}[label=(\alph*)]
			\item for all $\beta<\xi$, $N_\beta\subseteq\xi$;
			\item for all $\beta<\xi$ and $\beta'\in\omega_1\setminus\{\beta\}$, $N_{\beta}^-\cap N_{\beta'}^-=\emptyset$;
			\item for all $\beta<\omega_1$ and $i<\theta$, if $N_\beta\cap K_i\neq \emptyset$ then $\beta \in K_i$.
		\end{enumerate}
	\end{claim}
	\begin{proof} There are three cases to consider:
		
		$\br$ For every $\beta<\omega$, just set $N_\beta :=\{\beta\}$.
		
		$\br$ For every $\beta\in \omega_1\setminus(\xi+\omega)$, as $\alpha_\beta>\xi$, we may let
		$$\epsilon:=\max(\{\xi,\sup(A_\beta \cap K_0)\}\cap\alpha_\beta),$$
		and then set $N_{\beta} := N_{\beta}^{\epsilon+1}$.

		$\br$ Suppose $\beta$ is not of the above form. In particular, $A_\beta\in\mathcal B_\xi$
		and $f_\xi(A_\beta)<\alpha_\beta$.

		As $\vec K$ is discrete, let us pick an open neighborhood $U$ of $\beta$ which 
		for which $I:=\{ i<\theta\mid U\cap K_i\neq\emptyset\}$ contains at most one element.
		Find $\varepsilon<\alpha_\beta$ such that $A_{\beta}\setminus \varepsilon\subseteq U$,
		and then let
		$$\epsilon:=\begin{cases}
		\max(\{f_\xi(A_\beta),\varepsilon\}),&\text{if }I=\emptyset;\\
		\max(\{f_\xi(A_\beta),\varepsilon,\sup(A_{\beta}\cap K_i)\}\cap\alpha_\beta),&\text{if }I=\{i\}.
		\end{cases}$$
		Finally, set $N_{\beta} := N_{\beta}^{\epsilon+1}$.
		
		We omit the proof that $\langle N_\beta\mid\beta<\omega_1\rangle$ is as sought,
		because a similar verification is given in details in the proof of Claim~\ref{Dowker tails claim} below.
	\end{proof}	
	
	It now follows from Lemma~\ref{general normal lemma} that there exists a sequence $\langle U_i \mid i<\theta \rangle$ 
	of pairwise open sets 	such that $K_i\subseteq U_i$ for all $i<\theta$.
\end{proof}

\subsection{A Dowker space with small hereditary density}\label{Subsection - Dowker Space of size lambda^+ }

In \cite{Rudin_souslin_line_dowker_space}, Rudin constructed a Dowker space of size $\aleph_1$ from a Souslin tree.
In \cite{rudin_S_space_Souslin}, she constructed an $S$-space of size $\aleph_1$ from a Souslin tree. 
In \cite{rudin_separable_dowker_space}, she constructed an $S$-space of size $\aleph_1$ which is Dowker from a Souslin tree,
and in \cite{kappa_Dowker_from_souslin_Rudin}, she constructed a Dowker space of size $\lambda^+$ from a $\lambda^+$-Souslin tree, for $\lambda$ regular.

In \cite{de_Caux_space}, de Caux constructed a Dowker space of size $\aleph_1$ assuming $\clubsuit(\omega_1)$. Here we roughly follow de Caux's general approach,
but using $\clubsuit_{\ad}$, instead. So this gives a simultaneous generalization of de Caux's result and Rudin's result.

\begin{theorem} Suppose that $\clubsuit_{\ad}(\{E^{\lambda^+}_{\lambda}\},\lambda,1)$ holds for an infinite regular cardinal $\lambda$. 
Then there exists a collectionwise normal Dowker space $\mathbb X$ of cardinality $\lambda^+$
such that $\hd(\mathbb X)=\lambda$ and $\Lin(\mathbb X)=\lambda^+$.
\end{theorem}
\begin{remark} The preceding is the Introduction's Theorem~D.
\end{remark}

Let $\langle \mathcal A_{\alpha}\mid\alpha\in  E^{\lambda^+}_{\lambda}\rangle$ be a $\clubsuit_{\ad}(\{E^{\lambda^+}_{\lambda}\},\lambda,1)$-sequence.
For each $\alpha\in E^{\lambda^+}_{\lambda}$, fix an injective enumeration $\{ A^{j,n}_{\alpha+i}\mid i<\lambda, j\leq n<\omega\}$ of the elements of $\mathcal A_\alpha$.
For every $\xi\in\lambda^+\setminus\lambda$, as $\mathcal B_\xi:=\{ A^{j,n}_\beta\mid \lambda\le\beta<\xi+\lambda, j\le n<\omega\}$ is a subset of $\bigcup_{\alpha\in E^{\lambda^+}_\lambda}\mathcal A_\alpha$ of size $\lambda$,
we may appeal to Proposition~\ref{Proposition - disjointify multi-ladder system} to fix a function $f_\xi:\mathcal B_\xi\rightarrow\xi$ such that:
	\begin{enumerate}
	\item For every $B\in\mathcal B_\xi$, $f_\xi(B)\in B$;
	\item The sets in $\langle B\setminus f_\xi(B) \mid B\in\mathcal B_\xi\rangle$ are pairwise disjoint.
	\end{enumerate}		

\begin{definition} For every $\beta<\lambda^+$, let $\alpha_\beta:=\min\{\alpha\le\beta\mid \exists i<\lambda(\beta=\alpha+i)\}$.
\end{definition}

We are now ready to define a topology $\tau$ on the set $X:=\lambda^+\times\omega$.
For all $x\in \lambda\times \omega$, just let $\mathcal N_{x}:=\{\{x\}\}$.
For all $x=(\beta,n)$ in $X$ with $\beta\ge\lambda$, denote 
$N_{x}^\epsilon:=\{x\}\cup\bigcup_{j\leq n}((A^{j,n}_{\beta}\setminus\epsilon)\times\{j\})$, 
and then set  $\mathcal N_x:=\{ N_x^\epsilon\mid \epsilon<\alpha_\beta\}$.
Finally, a subset $ U \s \lambda^+\times\omega$ is $\tau$-open iff for any $x\in U $, 
there is $N\in\mathcal N_x$ with $N\s U$.	
It is easy to check that $\mathbb X:=(X, \tau) $ is a $ T_1 $ topological space.

\begin{definition}
	For any subset $N\s \lambda^+\times\omega$, denote:
	\begin{itemize}
\item	$N^-:=\{(\gamma,j)\in N\mid \exists (\beta,n)\in N\,(\gamma<\beta)\}$;
\item $N^j:=\{ \gamma\mid (\gamma,j)\in N^-\}$ for any $j<\omega$.
\end{itemize}
\end{definition}

The following is obvious.
\begin{lemma}\label{Fact - neighborhoods behave as expected}
	For all $x=(\beta,n)$ in $X$, $N\in\mathcal N_x$, and $j\leq n$, $N^j$ is a cofinal subset of $\alpha_\beta$ and $N\subseteq (\beta+1)\times(n+1)$.
In particular, for all $\delta<\lambda^+$ and $n<\omega$:
\begin{itemize}
\item $\delta\times n$ is $\tau$-open;
\item $\lambda^+\times(\omega\setminus n)$ is $\tau$-closed.\qed
\end{itemize}
\end{lemma}

So $\{ \delta\times\omega\mid \delta<\lambda^+\}$ witnesses that $\Lin(\mathbb X)=\lambda^+$.

\begin{lemma}\label{Dowker closure ost}
	For every $\alpha\in E^{\lambda^+}_{\lambda}$ and $k<\omega$, $$(\lambda^+\setminus(\alpha+\lambda))\times (\omega\setminus k)\s \cl([\alpha,\alpha+\lambda)\times\{k\}).$$
\end{lemma}
\begin{proof} Denote $I_{\alpha,j}:=[\alpha,\alpha+\lambda)\times\{j	\}$ and $F_{\alpha,j}:=\cl(I_{\alpha,j})$.
	\begin{claim}\label{Dowker closure ost - claim inner induction}
	Let $\alpha\in E^{\lambda^+}_{\lambda}$ and $j<\omega$. Then
		$(\lambda^+\setminus\alpha)\times \{j\}\s F_{\alpha,j}$.
	\end{claim}
	\begin{proof}
		We prove by induction on $ \delta\in E^{\lambda^+}_{\lambda}\setminus\alpha $ that $ I_{\delta,j}\subseteq F_{\alpha,j}$.
		
		$\br$ For $\delta=\alpha$, trivially $I_{\delta,j} \subseteq F_{\alpha,j}$.
			
		$\br$ Suppose that $\delta\in  E^{\lambda^+}_{\lambda}\setminus(\alpha+\lambda)$ is an ordinal such that, for every $\gamma\in E^\delta_\lambda\setminus\alpha$, $I_{\gamma,j}\s F_{\alpha,j}$.
		Therefore, $[\alpha,\delta)\times\{j\} \subseteq F_{\alpha,j}$. 
		Let $x\in  I_{\delta,j}$. Since, for all $N\in \mathcal N_x$, 
		$N^j$ is a cofinal subset of $\delta$, $N\cap( [\alpha,\delta)\times\{j\})\neq\emptyset$.
		Therefore, $x\in\cl([\alpha,\delta)\times\{j\})\s F_{\alpha,j}$.
	\end{proof}
	
	Let $\alpha\in E^{\lambda^+}_{\lambda}$ and $k<\omega$. Let $j\ge\kappa$ be some integer;
	we need to prove that $(\lambda^+\setminus(\alpha+\lambda))\times\{j\}\s F_{\alpha,k}$.
	By the preceding claim, it suffices to prove that $F_{\alpha+\lambda,j}\s F_{\alpha,k}$. 
	But the latter is a closed set, so it suffices to prove that $I_{\alpha+\lambda,j}\s F_{\alpha,k}$. 
	Let $x\in I_{\alpha+\lambda,j}$ be arbitrary. 
	For each $N\in \mathcal N_x$, as $k\le j$, $N^k$ is a cofinal subset of $\alpha+\lambda$,
	and then $N\cap I_{\alpha,k}\neq\emptyset$.
	Therefore, $x\in\cl(I_{\alpha,k})=F_{\alpha,k}$, as sought.
\end{proof}

\begin{cor}\label{corollary - Dowker space clubsuit_AD set of big size} For any $B\s X$ of size $\lambda^+$:
\begin{enumerate}
\item There exists $D\s B$ with $|D|=\lambda$ such that $B\s \cl(D)$;
\item\label{Clause - B closed contain a tail of the space} If $B$ is $\tau$-closed, then there is $(\beta,k)\in\lambda^+\times\omega$ such that $(\lambda^+\setminus\beta)\times(\omega\setminus k)\s B$.
\end{enumerate}
\end{cor}
\begin{proof} Given $B$ as above,
	fix some $k<\omega$ such that $|B\cap (\lambda^+\times \{k\})|=\lambda^+$. 
	By Clause~(\ref{clubsuit AD, Clause cofinal set}) of Definition~\ref{clubsuit AD definition}, 
	we can find an $ \alpha\in E^{\lambda^+}_{\lambda}$ such that $\dom((A^{j,n}_{\alpha+i}\times \{k \})\cap B)$ is cofinal in $\alpha$
	for all $ i<\lambda$ and $j\leq n<\omega$.
	Let $D:=B\cap(\alpha\times \{k\})$.
	If follows that, for every $x\in[\alpha,\alpha+\lambda)\times\{k\}$ and every $N\in\mathcal N_x$, 
	$\dom(N\cap D)$ is cofinal in $\alpha$. 
	So, $ [\alpha,\alpha+\lambda)\times\{k\} \subseteq \cl(D)$
	and hence $(\lambda^+\setminus(\alpha+\lambda))\times(\omega\setminus k) \s	\cl([\alpha,\alpha+\lambda)\times\{k\})\subseteq \cl(D)$.
	
	\begin{enumerate}
	\item As $|B\setminus\cl(D)|\le|X\setminus\cl(D)|\le\lambda$, we see that $D\cup(B\setminus\cl(D))$ is a dense subset of $B$ of cardinality $\lambda$.
	\item If $B$ is $\tau$-closed, then for $\beta:=\alpha+\lambda$,
	$(\lambda^+\setminus\beta)\times(\omega\setminus k)\s\cl(D)\s B$.\qedhere 
	\end{enumerate}
\end{proof}

\begin{cor}\label{lambda^+ Dowker space - no two disjoint closed uncountable sets}
	$\hd(\mathbb X)\le\lambda$ and there are no two disjoint closed subspaces of $\mathbb X$ of cardinality $\lambda^+$.
\end{cor}
\begin{proof}
By Corollary~\ref{corollary - Dowker space clubsuit_AD set of big size}\eqref{Clause - B closed contain a tail of the space}.
\end{proof}
\begin{lemma}\label{normality429} The space $\mathbb X$ is Hausdorff and collectionwise normal.
\end{lemma}
\begin{proof} As $\mathbb X$ is $T_1$, it suffices to verify that it is collectionwise normal.
	Fix a nonzero cardinal $\theta$ and a discrete sequence $\vec K=\langle K_i \mid i<\theta \rangle$ of closed sets. 
	It follows from Clause~(\ref{clubsuit AD, Clause cofinal set}) of Definition~\ref{clubsuit AD definition}
	that $\theta\le\lambda$, so, using Corollary~\ref{lambda^+ Dowker space - no two disjoint closed uncountable sets} and by possibly re-indexing, we 
	may find a large enough $\xi\in E^{\lambda^+}_\lambda$ such that $K_i\subseteq \xi\times\omega$ for all nonzero $i<\theta$.
	Recall that $O:=\xi \times \omega$ is an open set.
	
	\begin{claim}\label{Dowker tails claim}
		There exists a sequence $\langle N_x\mid x\in X\rangle\in\prod_{x\in X}\mathcal N_x$
		such that:
		\begin{enumerate}[label=(\alph*)]
			\item for all $x\in O$, $N_{x}\subseteq O$;
			\item\label{dowker disjoint tails} for all $x\in O$ and $x'\in X\setminus\{x\}$, $N_{x}^-\cap N_{x'}^-=\emptyset$;
			\item\label{dowker tails outside H,K} for all $x\in X$ and $i<\theta$, if $N_x\cap K_i\neq \emptyset$ then $x \in K_i$.
		\end{enumerate}
	\end{claim}	
	\begin{proof}  Recalling Lemma~\ref{Fact - neighborhoods behave as expected},
		we should only worry about requirements (b) and (c).
		Let $x=(\beta,n)$ in $X$.	There are three cases to consider:
	
		$\br$ If $\beta<\lambda$, then set $N_x :=\{x\}$. Evidently, requirement~(c) is satisfied.
		
		$\br$ If $\beta\ge\xi+\lambda$, then $\alpha_\beta>\xi$, so we let
		$$\epsilon:=\max(\{\xi,\sup(\dom[(\bigcup\nolimits_{j\leq n}( A^{j,n}_{\beta}\times \{j\}))\cap K_0])\}\cap\alpha_\beta),$$
		and then set $N_x := N_x^{\epsilon+1}$. Evidently, $N_x\cap O=\emptyset$.
		In particular, for all $i<\theta$, if $N_x\cap K_i\neq\emptyset$, then $i=0$ and $\sup(\dom[(\bigcup\nolimits_{j\leq n}( A^{j,n}_{\beta}\times \{j\}))\cap K_0])=\alpha_\beta$,
		which means that $x\in K_0$, since the latter is closed.
		So, requirement~(c) is satisfied.
		
		$\br$ If $\lambda\le\beta<\xi+\lambda$, then $A_\beta^{j,n}\in\mathcal B_\xi$ for all $j\le n$,
		so that $\phi_x:=\max\{ f_\xi(A^{j,n}_\beta)\mid j\le n\}$ is $<\alpha_\beta\le\xi$.
		As $\vec K$ is discrete, let us pick an open neighborhood $U$ of $x$ 
		for which $I:=\{ i<\theta\mid U\cap K_i\neq\emptyset\}$ contains at most one element,
		and then find a large enough $\varepsilon\in\alpha_\beta\setminus\phi_x$ such that $N^{\varepsilon}_x\subseteq U$.
		Let
		$$\epsilon:=\begin{cases}
		\varepsilon,&\text{if }I=\emptyset;\\
		\max(\{\varepsilon,\sup(\dom[(\bigcup_{j\leq n}( A^{j,n}_{\beta}\times \{j\}))\cap K_i])\}\cap\alpha_\beta),&\text{if }I=\{i\},
		\end{cases}$$
		and then set $N_x:=N_x^{\epsilon+1}$.
		Evidently, $N_x\s U$. So, for all $i<\theta$, if $N_x\cap K_i\neq\emptyset$, then $I=\{i\}$ and $\sup(\dom[(\bigcup\nolimits_{j\leq n}( A^{j,n}_{\beta}\times \{j\}))\cap K_i])=\alpha_\beta$,
		which means that $x\in K_i$. So, requirement~(c) is satisfied.
	
		We are left with verifying that $\langle N_x\mid x\in X\rangle$ satisfies requirement~(b).
		Fix arbitrary $x\in O$ and $x'\in X\setminus \{x\}$.
		Say, $x=(\beta,n)$ and $x'=(\beta',n')$.
		Note that if $\beta<\lambda$, then $N_x^-$ is empty,
		and if $\beta'\ge\xi+\lambda$, then $\min(\dom(N_{x'}))>\xi>\beta=\sup(\dom(N_x))$.
		Therefore, if $N_x^-\cap N_{x'}^-\neq\emptyset$, then $\lambda \le \beta,\beta'<\xi+\lambda$.
		Thus, assume that $\lambda \le \beta,\beta'<\xi+\lambda$, and fix $\epsilon\ge\phi_x$ and $\epsilon'\ge\phi_{x'}$ such that 		
		\begin{itemize}
\item		$N_{x}^-=\bigcup_{j\leq n}((A^{j,n}_{\beta}\setminus(\epsilon+1))\times\{j\})$, and
\item		$N_{x'}^-=\bigcup_{j\leq n'}((A^{j,n'}_{\beta'}\setminus(\epsilon'+1))\times\{j\})$.
\end{itemize}
So, if $N_x^-\cap N_{x'}^-\neq\emptyset$, then there exists $j\le\min\{n,n'\}$ such that $(A^{j,n}_{\beta}\setminus\phi_x)\cap (A^{j,n'}_{\beta'}\setminus\phi_{x'})\neq\emptyset$.
In particular, $(A^{j,n}_{\beta}\setminus f_\xi(A^{j,n}_{\beta}))\cap (A^{j,n'}_{\beta'}\setminus f_\xi(A^{j,n'}_{\beta'}))\neq\emptyset$,
contradicting the fact that $(\beta,n)\neq (\beta',n')$.
	\end{proof}	
	
	It now follows from Lemma~\ref{general normal lemma} that there exists a sequence $\langle U_i \mid i<\theta \rangle$
	of pairwise open sets such that $K_i\subseteq U_i$ for all $i<\theta$.
\end{proof}

\begin{lemma}
	The space $\mathbb X$ is Dowker.
\end{lemma}
\begin{proof}	
	Denote $D_n:=\lambda^+\times (\omega\setminus n)$. Notice that $\langle D_n \mid n<\omega \rangle $ is a $\subseteq$-decreasing sequence of $\lambda^+$-sized $\tau$-closed sets such that $\bigcap_{n<\omega} D_n = \emptyset$.
	By Corollary~\ref{lambda^+ Dowker space - no two disjoint closed uncountable sets}, there are no two disjoint closed sets of cardinality $\lambda^+$.
	So, by Corollary~\ref{normality429} and Lemma~\ref{Lemma - Dowker general argument}, the space $(X,\tau)$ is Dowker.
\end{proof}

\section{A Dowker space with a normal square}\label{normalsquare}

In \cite{diamond_omega_2_LS_doweker}, Szeptycki proved that, assuming $ \diamondsuit^*(S) $ for a stationary $S\s E^{\omega_2}_{\omega_1}$,
there exists a ladder-system over a subset of $S$ whose corresponding ladder-system space is a Dowker space having a normal square.
As seen in Section~\ref{sectionladdersystemspace}, the hypothesis may be reduced to $ \clubsuit_{\ad}(\{E^{\omega_2}_{\omega_1}\},1,2) $
and still give a ladder-system whose corresponding space is Dowker,\footnote{Recall that 
by Remark~\ref{diamondsuit iff clubsuit and ch} and Lemma~\ref{lemma216},
for $S\s E^{\omega_2}_{\omega_1}$, $\diamondsuit^*(S)\implies\diamondsuit(E^{\omega_2}_{\omega_1})\implies \clubsuit(E^{\omega_2}_{\omega_1})\implies \clubsuit_{\ad}(E^{\omega_2}_{\omega_1},\omega_1,{<}\omega)\implies \clubsuit_{\ad}(E^{\omega_2}_{\omega_1},1,2) $.}
since $E^{\omega_2}_{\omega_1}$ is a non-reflecting stationary subset of $\omega_2$.
But what about the normal square?

In this short section, we point out that the $\diamondsuit^*$ hypothesis may be reduced to an assumption in the language of the Brodsky-Rinot proxy principle (see Definition~\ref{proxyprinciple} below).
For the rest of this section, let $\lambda$ denote an infinite regular cardinal.
\begin{prop}\label{prop61} $\p_{\lambda}^-(\lambda^+,2,\sqleft{\lambda^+},\lambda,\{E^{\lambda^+}_{\lambda}\})$
entails the existence of a ladder-system over a subset of $E^{\lambda^+}_{\lambda}$ 
whose corresponding ladder-system space $(\lambda^+,\tau)$ is a Dowker space having a normal square.
\end{prop}

The point is that the hypothesis of the preceding already follows
from $\diamondsuit(E^{\lambda^+}_{\lambda})$, but it is also consistent with its failure (see Clauses (1) and (11) of \cite[Theorem~6.1]{paper23}).

\begin{definition}[Brodsky-Rinot, \cite{paper23}]\label{proxyprinciple}  For a family $\mathcal S\s\mathcal P(\kappa)$,
	and a cardinal $\theta<\kappa$,
	$\p_\xi^-(\kappa,2,\sqleft{\kappa},\theta,\mathcal S)$ asserts the existence of a sequence $\langle C_\alpha\mid\alpha<\kappa\rangle$
	such that:
	\begin{itemize}
		\item for every $\alpha\in\acc(\kappa)$, $C_\alpha$ is a club in $\alpha$ of order-type $\le\xi$;
		\item for every $S\in\mathcal S$ and every sequence $\langle B_i \mid i < \theta \rangle$ of cofinal subsets of $\kappa$,
		there exist stationarily many $\alpha \in S$ such that, for all $i < \theta$,
		$$\sup\{\delta\in B_i\cap\alpha\mid \min(C_\alpha\setminus(\delta+1))\in B_i\}=\alpha.$$
	\end{itemize}
\end{definition}

Note that for every $\mathcal S\s\mathcal P(E^{\lambda^+}_{\lambda})$, any $\p_{\lambda}^-(\lambda^+,2,\sqleft{\lambda^+},\theta,\mathcal S)$-sequence
witnesses the validity of $\clubsuit_{\ad}(\mathcal S,1,\theta)$.

\begin{fact}[Brodsky-Rinot, {\cite[Theorem~4.15]{paper23}}]\label{Fact Good proxy}  For a family $\mathcal S\s\mathcal P(\kappa)$,
	and a cardinal $\theta<\kappa$,
	$\p_\xi^-(\kappa,2,\sqleft{\kappa},\theta,\mathcal S)$ entails the existence of a sequence $\langle C_\alpha\mid\alpha<\kappa\rangle$
	such that:
	\begin{enumerate}
		\item for every $\alpha\in\acc(\kappa)$, $C_\alpha$ is a club in $\alpha$ of order-type $\le\xi$;
		\item\label{Fact Good proxy - Clause hitting}  for every $S\in\mathcal S$ and every sequence $\langle \mathcal B_i \mid i < \theta \rangle$ with $\mathcal B_i\s[\kappa]^{<\omega}$ and $\mup(\mathcal B_i)=\kappa$ for all $i<\theta$,
		there exist stationarily many $\alpha \in S$ such that, for all $i<\theta$,
		$$\mup\{ x\in\mathcal B_i\mid x\s C_\alpha\}=\alpha.$$ 
	\end{enumerate}
\end{fact}

For an ordinal $\alpha$, let us say that a subset $x$ of the product $\alpha\times\alpha$ is \emph{dominating}
iff for every $(\beta,\gamma)\in(\alpha,\alpha)$, there exists $(\beta',\gamma')\in x$ with $\beta\le\beta'$ and $\gamma\le\gamma'$.
Now, we are ready to prove the main lemma.

\begin{lemma}\label{lemma - ladder-system club sequence with square}
Suppose that $\p_{\lambda}^-(\lambda^+,2,\sqleft{\lambda^+},\lambda,\{E^{\lambda^+}_{\lambda}\})$ holds.
Then there exist a partition $\langle S_n\mid n<\omega\rangle$ of $\lambda^+$ into stationary sets and a sequence $\langle s_\alpha\mid \alpha\in E^{\lambda^+}_\lambda \rangle $ such that:
	\begin{enumerate}
		\item\label{lemma - ladder-system club sequence with square - Clause type} For each $\alpha\in E^{\lambda^+}_\lambda$, $s_\alpha$ is either empty or a cofinal subset of $\alpha$ of order-type $\lambda$;
		\item\label{lemma - ladder-system club sequence with square - Clause bigcup S_i is open} For all $n<\omega$ and $\alpha\in S_{n+1}$,  $s_\alpha\subseteq \bigcup_{i\leq n} S_i$;
		\item\label{lemma - ladder-system club sequence with square - Clause hitting} For every $k<\omega$, every $\lambda$-sized subfamily $\mathcal F \subseteq [\bigcup_{i\leq k} S_i]^{\lambda^+}$, 
		and every $n\in\omega\setminus(k+1)$, the following set is stationary:
		$$ \{ \alpha\in S_n \mid \forall F\in\mathcal F~[ \sup(s_\alpha\cap F)=\alpha]  \};$$
		\item\label{lemma - ladder-system club sequence with square - Clause square hitting} 
		For every two dominating subsets $B_0,B_1$ of  $\lambda^+\times\lambda^+$, there exists $m<\omega$ such that for every $n\in\omega\setminus m$ the following set is stationary:
		$$ \{ \alpha\in S_n \mid \forall i<2~[(s_\alpha\times s_\alpha)\cap B_i\text{ dominates }(\alpha,\alpha)] \}.$$
	\end{enumerate}
\end{lemma}
\begin{proof} 
Let $\vec C=\langle C_\alpha\mid\alpha<\lambda^+\rangle$ be a $\p_{\lambda}^-(\lambda^+,\lambda^+,\sqleft{\lambda^+},\lambda,\{E^{\lambda^+}_\lambda\})$-sequence.
Let $\mathcal I$ denote the collection of all $T\s E^{\lambda^+}_\lambda$
such that $\vec C$ is not a $\p_{\lambda}^-(\lambda^+,\lambda^+,\sqleft{\lambda^+},\lambda,\{T\})$-sequence.
Evidently, $\mathcal I$ is a $\lambda^+$-complete ideal over $E^{\lambda^+}_\lambda$. So, by Ulam's theorem,
$\mathcal I$ is not weakly $\lambda^+$-saturated.
This means that exists a sequence $\langle S_\iota\mid \iota<\lambda^+\rangle$ of pairwise disjoint subsets of $E^{\lambda^+}_\lambda$,
such that, for each $\iota<\lambda^+$, 
$\vec C$ is a $\p_{\lambda}^-(\lambda^+,\lambda^+,\sqleft{\lambda^+},\lambda,\{S_\iota\})$-sequence.
In particular, we may fix a family $\mathcal S$ consisting of $\aleph_0$-many pairwise disjoint stationary subsets of $E^{\lambda^+}_{\lambda}$
such that $\p_{\lambda}^-(\lambda^+,\lambda^+,\sqleft{\lambda^+},\lambda,\mathcal S)$ holds and yet $S_0:=\lambda^+\setminus\bigcup\mathcal S$ is stationary.
Fix an injective enumeration $\langle S_{n+1}\mid n<\omega\rangle$ of $\mathcal S$.
	For every $\alpha\in E^{\lambda^+}_{\lambda}$, let $n(\alpha)$ be such that $\alpha\in S_{n(\alpha)}$.
	For each $n<\omega$, let $ W_n:=\bigcup_{i\leq n}S_i $. 
Now, let $\vec D= \langle D_\alpha\mid\alpha<\lambda^+\rangle$ be a	 $\p_{\lambda}^-(\lambda^+,\lambda^+,\sqleft{\lambda^+},\lambda,\{S_{n+1} \mid n<\omega \})$-sequence as in Fact \ref{Fact Good proxy}.
	For every $\alpha\in E^{\lambda^+}_{\lambda}$, let

$$s_\alpha:=\begin{cases}
W_{n(\alpha)-1}\cap D_\alpha,&\text{if }n(\alpha)>0\ \&\ \sup(W_{n(\alpha)-1}\cap D_\alpha)=\alpha;\\
\emptyset,&\text{otherwise.}
\end{cases}$$

	We claim that the sequence $\langle s_\alpha \mid \alpha \in E^{\lambda^+}_{\lambda}\rangle$ is as sought.
	Notice that Clauses (\ref{lemma - ladder-system club sequence with square - Clause type}) and (\ref{lemma - ladder-system club sequence with square - Clause bigcup S_i is open}) 
	hold by our very construction, as $D_\alpha$ has order-type $\lambda$ for every $\alpha\in E^{\lambda^+}_{\lambda}$.
	\begin{claim} 	Let $k<\omega$ and $\mathcal F \subseteq [W_k]^{\lambda^+}$ be a family of size $\lambda$.	
	For every integer $n>k$, $\{ \alpha\in S_n \mid \forall F\in\mathcal F~[ \sup(s_\alpha\cap F)=\alpha]  \}$ is stationary.
	\end{claim}
	\begin{proof} Let $\{\mathcal B_i\mid i<\lambda\}$ be some enumeration of $\{ [F]^1 \mid F\in\mathcal F\}$.
	Evidently, for every $i<\lambda$, $\mathcal B_i\subseteq[\lambda^+]^{<\omega}$ and $\mup(\mathcal B_i)=\lambda^+$.
	Now, by Fact~\ref{Fact Good proxy}, Clause~(\ref{Fact Good proxy - Clause hitting}),
	for every $n<\omega$, the set $G_n$ of all $\alpha \in S_n$ such that, for all $i<\lambda$,
		$$\mup\{ x\in\mathcal B_i\mid x\subseteq D_\alpha\}=\alpha,$$
		is stationary.
		In particular, for all $n>k$, $\alpha\in G_n$ and $F\in\mathcal F$:
		\[\sup(s_\alpha\cap F)=\sup(W_{n(\alpha)-1}\cap D_\alpha\cap F)=\alpha.\qedhere\]
	\end{proof}
	\begin{claim} 	Let $B_0,B_1$ be two dominating subsets of $\lambda^+\times\lambda^+$. 
	 Then there exists some $m<\omega$ such that for every $n\in\omega\setminus m$ the following set is stationary:
		$$ \{ \alpha\in S_n \mid \forall i<2~[(s_\alpha\times s_\alpha)\cap B_i\text{ dominates }(\alpha,\alpha)] \}.$$
    \end{claim}
    \begin{proof} For every $\epsilon<\lambda^+$ and $i<2$, fix $(\xi^i_\epsilon,\zeta^i_\epsilon)\in B_i$ with $\min\{\xi^i_\epsilon,\zeta^i_\epsilon\}>\epsilon$.
    For every $n<\omega$, let 
    $$\mathcal B_n:=\{\{\xi^0_\epsilon,\zeta^0_\epsilon,\xi^1_\epsilon,\zeta^1_\epsilon\}\mid \epsilon<\lambda^+, \max\{n(\xi^0_\epsilon),n(\zeta^0_\epsilon),n(\xi^1_\epsilon),n(\zeta^1_\epsilon)\}=n\}.$$    
    By the pigeonhole principle, we may find $m<\omega$ such that $|\mathcal B_{m}|=\lambda^+$. In particular, 
    $\mathcal B_{m}\s[\lambda^+]^{<\omega}$ and $\mup(\mathcal B_{m})=\lambda^+$.
	Now, by Fact~\ref{Fact Good proxy}, Clause (\ref{Fact Good proxy - Clause hitting}),
	for every $n<\omega$, the set $G_n$ of all $\alpha \in S_n$ such
		$$\mup\{ x\in\mathcal B_m\mid x\subseteq D_\alpha\}=\alpha,$$
		is stationary.
		In particular, for all $n>m$, $\alpha\in G_n$ and $i<2$, 
		$(s_\alpha)^2\cap B_i=(D_\alpha\cap W_{n(\alpha)-1})^2\cap B_i$ and it dominates $(\alpha,\alpha)$.    
    \end{proof}
This completes the proof.
\end{proof}

At this point, the proof of Proposition~\ref{prop61} continues exactly as in \cite[Theorem~4]{diamond_omega_2_LS_doweker},
using the sequences $\langle S_n\mid n<\omega\rangle$ and $\langle s_\alpha\mid\alpha<\lambda^+\rangle$ constructed in the preceding lemma.

\begin{remark} Lemma~\ref{lemma - ladder-system club sequence with square} also implies that 
$\p_{\omega}^-(\omega_1,2,\sqleft{\omega_1},\omega,\{\omega_1\})$ is a weakening of $\clubsuit^*$ sufficient for the constructions of \cite{MR1429177}.
\end{remark}

\section{Acknowledgements}
Some of the results of this paper come from the second author's M.Sc.~thesis written under the supervision of the first author at Bar-Ilan University.
We are grateful to Bill Weiss for kindly sharing with us a scan of Dahroug's handwritten notes with the construction of an Ostaszewski space from a Souslin tree and $\ch$.
Our thanks go to Tanmay Inamdar for many illuminating discussions,
and to Istv\'an Juh\'asz for reading a preliminary version of this paper and providing a valuable feedback.
We also thank the  referee for a useful feedback.

Both authors were partially supported by the Israel Science Foundation (grant agreement 2066/18).
The first author was also partially supported by the European Research Council (grant agreement ERC-2018-StG 802756).

\newcommand{\etalchar}[1]{$^{#1}$}

\end{document}